\documentclass[12pt]{article}
\usepackage[a4paper, total={17cm,25cm}]{geometry}

\usepackage{amsmath,amssymb,amsfonts}
\usepackage{graphicx}
\usepackage{caption}
\usepackage{lipsum}

\usepackage[alphabetic,y2k,initials]{amsrefs}
\usepackage{pst-all}
\usepackage{pstricks-add}
\usepackage{pst-slpe}
\usepackage{pst-solides3d}

\usepackage[all]{xy}
\usepackage{textcomp}

\usepackage{calrsfs}
\DeclareMathAlphabet{\pazocal}{OMS}{zplm}{m}{n}
\newcommand{\LL}{\mathcal{L}}
\newcommand{\Q}{\pazocal{Q}}
\newcommand{\E}{\pazocal{E}}
\newcommand{\I}{\pazocal{I}}
\newcommand{\Pol}{\mathcal{P}}
\newcommand{\T}{\mathcal{T}}
\newcommand{\Set}{\mathcal{S}}
\newcommand{\Curve}{\pazocal{C}}

\def\rank{\mathrm{rank}}

\def\deg {\mathrm{deg\,}}

\newtheorem{theorem}{Theorem}[section]
\newtheorem{proposition}[theorem]{Proposition}
\newtheorem{corollary}[theorem]{Corollary}
\newtheorem{lemma}[theorem]{Lemma}
\newtheorem{remark}[theorem]{Remark}
\newtheorem{example}[theorem]{Example}

\newenvironment{proof}[1][Proof]{\noindent\textit{#1.} }{\hfill$\Box$\newline\medskip}

\usepackage{authblk}
\author[1,3]{Vladimir Dragovi\'c}
\author[2,3]{Milena Radnovi\'c}
\affil[1]{\textsc{The University of Texas at Dallas, Department for Mathematical Sciences}\newline\texttt{vladimir.dragovic@utdallas.edu}}
\affil[2]{\textsc{The University of Sydney, School of Mathematics and Statistics}\newline\texttt{milena.radnovic@sydney.edu.au}}
\affil[3]{\textsc{Mathematical Institute SANU, Belgrade}}

\date{}

\title{Periodic ellipsoidal billiard trajectories and extremal polynomials}

\usepackage{tikz}

\begin{document}

\maketitle

\begin{abstract}
A comprehensive study of periodic trajectories of billiards within ellipsoids in $d$-dimensional Euclidean space is presented.
The novelty of the approach is based on a relationship established between periodic billiard trajectories and extremal polynomials on the systems of $d$ intervals on the real line.
By leveraging deep, but yet not widely known results of the Krein-Levin-Nudelman theory of generalized Chebyshev polynomials, fundamental properties of billiard dynamics are proven for any $d$, viz., that the sequences of winding numbers are monotonic.   By employing the potential theory we prove the injectivity of the frequency map.
As a byproduct, for $d=2$ a new proof of the monotonicity of the rotation number is obtained.
The case study of trajectories of small periods $T$, $d\le T\le 2d$ is given.
In particular, it is proven that all $d$-periodic trajectories are contained in a coordinate-hyperplane and that for a given ellipsoid, there is a unique set of caustics which generates $d+1$-periodic trajectories.
A complete catalog of billiard trajectories with small periods is  provided for $d=3$.
\end{abstract}

\section{Introduction}\label{sec:intro}
Our aim in this paper is to develop a strong link between the theory of billiards within quadrics in $d$-dimensional space  and the theory of approximation, in particular the extremal polynomials on the systems of $d$ intervals on the real line. This link appears to be fruitful and enables  us to prove fundamental properties of the billiard dynamics and to provide a comprehensive study of periodic trajectories of the billiards within ellipsoids in the $d$-dimensional Euclidean space.

It is well-known that the billiard systems within  ellipsoids are integrable.
It seems that is also widely accepted that the integrable systems, understood as a quest for the exact solutions, are much less related to the approximation theory than their nonintegrable counterparts. It is worth remarking that these two streams of ideas that are being merged in this paper are rooted back in the first half of the XIX century in the works
of the same person -- Jean Victor Poncelet.

\subsection{Poncelet the engineer and Poncelet the geometer}

Modern algebraic approximation theory, especially extremal polynomials and  continued fraction theory
was established by Chebyshev along with his Sankt Petersburg school in the second half of the XIX century, and their followers in the modern times.
Highlights of these studies included the discovery of the Chebyshev polynomials, their generalizations to the systems of intervals, the points of alternance and the theorem of alternance, continued fractions, to mention a few.  All of this is  employed by us in the present work.

Chebyshev's motivation for these studies was his interest in practical problems: in mechanism theory to estimate the error of the mechanisms which transform linear motion into circular, like \emph{ Watt's complete parallelogram}.
The starting point of Chebyshev's investigation (\cite{Tcheb1852}) was work on the theory of mechanisms of the French military engineer,
professor of mechanics and academician Jean Victor Poncelet.
Poncelet came to the question of rational and linear approximation of the functions $\sqrt{X_2(x)}$ of the form of the square root of quadratic polynomials, and he gave two approaches to the problems he encountered -- one based on  analytical arguments and the second one based on geometric considerations.

Chebyshev learnt about this work of Poncelet during his trip abroad in 1852, although they did not meet in person.
However, Chebyshev did meet Cayley, who was at that time interested in another problem which originated again from Poncelet, this time -- Poncelet the geometer.
Upon return, Chebyshev described Poncelet  as a ``well-known scientist in practical mechanics'' (see \cite{Tcheb1852}).
Nowadays J.~V.~Poncelet is known first of all as one of the major geometers of the XIX century.

\subsection{Cayley's condition for the Poncelet Theorem}\label{sec:def}

Let $C$ and $\Gamma$ be two conics in the projective plane.
The question of interest is whether there  exists an $n$-polygon inscribed in $\Gamma$ and circumscribed about $C$.
The Poncelet theorem (\cite{Poncelet1822}, see also \cites{LebCONIQUES, GrifHar1978, DragRadn2011book, FlattoBOOK}) states that if such a polygon exists, there are infinitely many such polygons, and any point of the boundary $\Gamma$ is a vertex of one of them.

Denote by the same letters, $C$ and $\Gamma$, the symmetric $3\times3$ matrices such that the conics are given by the equations
$\langle Cz, z\rangle=0$ and $\langle\Gamma z, z\rangle=0$ in  projective coordinates.
Let $d_3(x)=\det (C + x \Gamma)$ be  the discriminant of the conic $C + x \Gamma = 0$ from the pencil generated by $\Gamma$ and $C$.
For $C$ and $\Gamma$ in a general position, $d_3$ is a cubic polynomial with no multiple roots.
Cayley (\cite{Cayley1854}) reduced the question of the existence of an $n$-polygon inscribed in $\Gamma$ and circumscribed about $C$ to the question whether the points $(0, \pm\sqrt{d_3(0)})$ are of order $n$ on the cubic curve $y^2=d_3(x)$.

If one analyzes carefully the approaches to the last question both in classics (see for example \cite{LebCONIQUES}) or in contemporary texts (see \cites{GrifHar1978,DragRadn2006jmpa,DragRadn2011book}), one may see that it reduces further to the existence of polynomials $q(x)$ and $p(x)$ of degrees $[\frac{n-1}2]-1$ and $[\frac{n}2]$ such that the function
$
\varphi(x)=q(x)\sqrt{d_3(x)} + p(x)
$
has a zero of multiplicity $n$ at $x=0$, i.e.
$$
\varphi(0)=\varphi'(0)=\dots=\varphi^{(n-1)}(0)=0.
$$
For $n=2m$, we get from the expanded expression:
$$
\varphi(x)=\sqrt{d_3(x)}(a_0 x^{m-2} + a_1 x^{m-3} + \dots + a_{m-2}) +
    (b_0 x^m + b_1 x^{m-1} + \dots + b_m)
$$
that there is an $n$-polygon inscribed in $\Gamma$ and circumscribed about $C$ if and only if it is possible to find a non-trivial set of coefficients $a_0$, $a_1$, \dots\ such that:
$$
\begin{array}{ccccccccc}
    a_0 C_3     & + & a_1 C_4     & + & \dots & + & a_{m-2} C_{ m+1} & = & 0 \\
    a_0 C_4     & + & a_1 C_5     & + & \dots & + & a_{m-2} C_{ m+2} & = & 0 \\
    \dots  \\
    a_0 C_{m+1} & + & a_1 C_{m+2} & + & \dots & + & a_{m-2} C_{2m-1} & = & 0,
\end{array}
$$
where
$\sqrt{d_3(x)} = C_0 + C_1x + C_2 x^2 + C_3 x^3 + \dots.$
Finally, for $n=2m$, we obtain \emph{the Cayley's condition}
\begin{equation}\label{eq:Cayodd} \left | \begin{array}{llll}
     C_3     & C_4     & \dots & C_{m+1} \\
     C_4     & C_5     & \dots & C_{m+2} \\
      & & \dots                          \\
     C_{m+1} & C_{m+2} & \dots & C_{2m-1}
     \end{array} \right |=0.
     \end{equation}
Similarly, for $n = 2m+1$, we obtain:
\begin{equation}\label{eq:Cayleyeven} \left | \begin{array}{llll}
     C_2     & C_3     & \dots & C_{m+1} \\
     C_3     & C_4     & \dots & C_{m+2} \\
      & & \dots                         \\
     C_{m+1} & C_{m+2} & \dots & C_{2m}
     \end{array} \right |=0.
\end{equation}

\subsection{Pad\'e approximation}

Halphen observed the significance of the polynomials $p$, $q$ and their relationship to the important questions of rational approximation of elliptic functions 130 years ago, while he was developing further the theory of rational approximation and of continued fractions of square roots $\sqrt {X_4(x)}$ of polynomials of degree up to four. This theory  was initiated by Abel and Jacobi.

Let us mention that so-called Pad\'e approximants play important role in the theory of rational approximations.
Consider a power series
$$
f(x)= C_0 + C_1x + C_2 x^2 + C_3 x^3 + \dots,
$$
and non-negative integers $k$, $l$.
\emph{A Pad\'e approximant of type $(k,l)$ of $f$} is a pair of polynomials $p_k$, $q_l$ such that
$$
\deg p_k\le k,
\quad
\deg q_l\le l,
\quad
(q_lf-p_k)(x)=\mathcal O(x^{k+l+1}).
$$
The index $(k,l)$ is said to be \emph{normal} for the Pad\'e table if $\deg p_k=k$ and $\deg q_l= l$.
The normality criterion can be reformulates as $H_{k,l}H_{k, l+1}H_{k+1,l}\ne 0$, with  \emph{the Hadamard-Hankel determinants} denoted as:
$$ H_{k,l}:=\left | \begin{array}{llll}
     C_{k-l+1}     & C_{k-l+2}    & \dots & C_{k} \\
     C_{k-l+2}     & C_{k-l+3}     & \dots & C_{k+1} \\
      & & \dots                         \\
     C_{k} & C_{k+1} & \dots & C_{k+l-1}
     \end{array} \right |.
$$

Halphen established a relationship between the Poncelet polygons and continued fractions and approximation theory.
We provide in our terminology his result from  \cite{Hal1888}*{Part 2, page 600},  which was until recently mostly forgotten:

\begin{theorem}[Halphen (1888)]\label{th:halphen}
There exists an $n$-gon inscribed in $\Gamma$ and circumscribed
about $C$ if and only if the elliptic function $y=\sqrt{d_3(x)}$ satisfies the following:
\begin{itemize}
\item [(a)] for $n=2m$ has a $(m+1,m-2)$  Pad\'e approximant with polynomials $p_m$ and $q_{m-2}$ of degrees $m, m-2$ respectively, which is not normal:

    $$ p_m(x)+q_{m-2}(x)\sqrt{d_3(x)} =\mathcal{O}(x^{2m}). $$
 \item [(b)] for $n=2m+1$ has a  $(m+1, m-1)$  Pad\'e approximant with polynomials $p_m$ and $q_{m-1}$ of degrees $m, m-1$ respectively, which is not normal:

    $$ p_m(x)+q_{m-1}(x)\sqrt{d_3(x)} =\mathcal{O}(x^{2m+1}).$$
    \end{itemize}
    \end{theorem}

    \begin{remark} The Cayley condition is $H_{m+1, m-1}=0$ for $n=2m$ and $H_{m+1,m}=0$ for $n=2m+1$.
     \end{remark}

\subsection{The overview of the results of the paper}

In the modern literature about Poncelet polygons and their higher dimensional generalizations, the polynomials $p$, $q$ have explicitly  appeared in several places, see for example \cites{DragRadn2006jmpa, DragRadn2011book}, but there they were not specifically emphasized and studied.
As a recent exception, we should mention \cite{RR2014}, where certain polynomial representations were studied and three important conjectures were formulated, see Remark \ref{rem:ramirez}.
On the other hand the extremal nature of the polynomials has not been observed until the present work, where we establish a fundamental connection between periodic ellipsoidal billiard trajectories and related extremal polynomials.
From their interplay we obtain essential results concerning the billiard frequency maps: the monotone nature of the winding numbers in Theorem \ref{th:winding} and the one-to-one property of the frequency map in Theorem \ref{th:1-1a}.

Periodic trajectories of ellipsoidal billiards and the corresponding frequency maps were also studied in \cites{PT2011, CRR2011,CRR2012} while yet another direction of the development of Halphen's ideas towards hyperelliptic functions was suggested in \cite{Drag2009}.
Ellipsoidal billiards have been  intensively studied in various frameworks in recent years, see for example \cites{Wiersig2000,DelshamsFedRR2001,Fed2001,BDFRR2002,WDullin2002,AbendaFed2004,KhTab2009,AbendaGrin2010,JovJov2014,Radn2015,IT2017} and references therein.

\smallskip

This paper is organized as follows.
In Section \ref{sec:hyper} we revisit the connection between the periodic trajectories of ellipsoidal billiards and finite order divisors on the Jacobians of hyper-elliptic curves.
We recall the analytic conditions for trajectories periodic in  elliptic coordinates (Theorem \ref{th:genCayley}) and show that periodicity is equivalent to the existence of polynomial solutions of certain functional equations, see Proposition \ref{prop:cayley-poly}.
Then in Theorem \ref{th:uslov} we give precise algebro-geometric conditions for periodicity involving the types of caustics.
In Section \ref{sec:low}, the case study of the trajectories of small periods $T$, $d\le T\le 2d$ is initiated.
It is proven in Theorem \ref{th:d} that all $d$-periodic trajectories are contained in a coordinate-hyperplane and that the trajectories of small periods must have certain number of pairs of caustics of the same type, see Theorem \ref{th:d+k}.
In particular, we proved in Theorem \ref{th:cayley} the uniqueness of the types of caustics for $(d+1)$-periodic trajectories.
Section \ref{sec:poly} explores in depth the connection with Krein-Levin-Nudelman theory of the generalized Chebyshev polynomials.
We prove the fundamental properties of  billiard dynamics for any dimension $d$ by answering positively all the conjectures from \cite{RR2014}: in Theorem \ref{th:winding} we prove that the winding numbers are strictly decreasing and derive the exact relationship between them and the signature, and in Corollary \ref{cor:Q} that all the zeroes of the corresponding polynomials $\hat{q}_{n-d}$ are real.
In Section \ref{sec:freq},
 we show that the frequency maps are injective over rational values, see Theorem \ref{th:signature-caustics}. 
  By employing the potential theory we prove injectivity of the frequency map -- see Theorem \ref{th:1-1a}.
As a byproduct, we obtain a new proof of the monotonicity of the rotation number in the $2$-dimensional case.
In Section \ref{sec:examples}, we provide a complete catalog of billiard trajectories with small periods $T\le6$ in the three-dimensional case and discuss their properties.

\section{Periodic trajectories and finite order divisors on hyper-elliptic curves}
\label{sec:hyper}

\subsection{Billiards within ellipsoids. Winding numbers. Elliptic periods.}

If the conics $C$ and $\Gamma$ from Section \ref{sec:def} are assumed to be confocal conics in the Euclidean plane, then the Poncelet polygons transform to periodic billiard trajectories within $\Gamma$.
Thus higher-dimensional generalizations of Poncelet polygons are related to periodic trajectories of billiards within quadrics.
In this section, we will discuss algebro-geometric conditions for periodicity of billiard trajectories within an ellipsoid in the $d$-dimensional space.

We note that a variety of higher-dimensional analogues of Poncelet polygons were introduced in \cites{DarbouxSUR, CCS1993, DragRadn2006jmpa, DragRadn2008}, see also \cite{DragRadn2011book} for a systematic exposition and bibliography therein.
Corresponding Cayley-type conditions were derived by the present authors in \cites{DragRadn1998a, DragRadn1998b, DragRadn2004, DragRadn2008}, see also \cite{DragRadn2011book}.

Let an ellipsoid be given by:
$$
\E\ :\ \frac{x_1^2}{a_1}+\dots+\frac{x_d^2}{a_d}=1,
\quad
0<a_1<a_2<\dots<a_d.
$$
The family of quadrics confocal with $\E$ is:
\begin{equation}\label{eq:confocald}
\Q_{\lambda}(x)=\frac {x_1^2}{a_1-\lambda}+\dots + \frac
{x_d^2}{a_d-\lambda}=1.
\end{equation}

For a point given by its Cartesian coordinates $x=(x_1, \dots, x_d)$, the Jacobi elliptic coordinates $(\lambda_1,\dots, \lambda_d)$ are defined as the solutions of the equation in $\lambda$: $ \Q_{\lambda}(x)=1$.
The correspondence between the elliptic and Cartesian coordinates is not injective, since points symmetric with respect to the coordinate hyper-planes have equal elliptic coordinates.

The Chasles theorem states that almost any line $\ell$ in the Euclidean space $\mathbf E^d$ is tangent to exactly $d-1$ non-degenerate 
quadrics from the confocal family.
Moreover, any line $\ell'$ obtained by a billiard reflection off a quadric from the confocal family \eqref{eq:confocald}
is touching the same $d-1$ quadrics, and consequently all segments of a given billiard trajectory within a quadric will by tangent to the same set of $d-1$ quadrics confocal with the boundary.
Those $d-1$ quadrics are called \emph{caustics} of the trajectory.  
The existence of caustics is a geometric manifestation of integrability of billiards within quadrics.
If those quadrics are $\Q_{\alpha_1}$, \dots, $\Q_{\alpha_{d-1}}$, then
the Jacobi elliptic coordinates $(\lambda_1,\dots, \lambda_d)$ of any point on $\ell$ satisfy the
inequalities $\Pol(\lambda_j)\ge 0$, $j=1,\dots,d$, where
$$
\Pol(x)=(a_1-x)\dots(a_d-x)(\alpha_1-x)\dots(\alpha_{d-1}-x).
$$
Let $b_1<b_2<\dots<b_{2d-1}$ be constants such that
$$\{b_1,\dots,b_{2d-1}\}=\{a_1,\dots,a_d,\alpha_1,\dots,\alpha_{d-1}\}.$$
Here, clearly, $b_{2d-1}=a_d$.
The possible arrangements of the parameters $\alpha_1$, \dots, $\alpha_{d-1}$ of the caustics and the parameters $a_1$, \dots, $a_d$ of the confocal family can be obtained from the following lemma.

\begin{lemma}[\cite{Audin1994}]\label{lemma:audin}
If $\alpha_1<\alpha_2< \dots<\alpha_{d-1}$, then
$\alpha_j\in\{b_{2j-1},b_{2j}\}$, for $1\le j\le d-1$.
\end{lemma}

If $\ell$ is the line containing a segment of a billiard trajectory within $\E$, then $b_1>0$.

Along a billiard trajectory, the Jacobi elliptic coordinates satisfy:
$$
b_0=0\le\lambda_1\le b_1,
\quad
b_2\le\lambda_2\le b_3,
\quad\dots,\quad
b_{2d-2}\le\lambda_{d}\le b_{2d-1}.
$$
Moreover, along the trajectory, each Jacobi coordinate $\lambda_j$ fills the whole interval $[b_{2j-2},b_{2j-1}]$, with local extreme points being only the end-points of the interval.
Thus, $\lambda_j$ takes values $b_{2j-2}$ and $b_{2j-1}$ alternately and changes monotonously between them.
Let $\T$ be a periodic billiard trajectory and denote by $m_j$ the number of its points where $\lambda_j=b_{2j-2}$.
Based on the previous discussion, $m_j$ is also the number of the points where $\lambda_j=b_{2j-1}$.

Notice that the value $\lambda_1=0$ corresponds to an impact with the boundary ellipsoid $\E$,
value $\lambda_j=\alpha_{k}$ corresponds to a tangency with the caustic $\Q_{\alpha_k}$, and
$\lambda_j=a_k$ corresponds to an intersection with the coordinate hyperplane $x_k=0$.
Since each periodic trajectory intersects any hyperplane even number of times, we get that $m_j$ must be even whenever $b_{2j-2}$ or $b_{2j-1}$ is in the set $\{a_1,\dots, a_d\}$.

Following \cite{RR2014}, we denote $m_0=n$, $m_d=0$, and call
$(m_0, m_1, \dots, m_{d-1})$
\emph{the winding numbers} of a given periodic billiard trajectory with period $n$. In addition we introduce \emph{the elliptic period} $\tilde n$ as the number of impacts after which the trajectory closes in the Jacobi elliptic coordinates. If
$$k=\gcd(m_0, m_1, \dots, m_{d-1}),$$ then $\tilde n=m_0/k$;
in addition $\tilde m_i=m_i/k$ are \emph{the elliptic winding numbers}.

\subsection{Hyperelliptic curves and periodic billiard trajectories}

We will use the following notation for the hyperelliptic curve and points on it:
\begin{equation}\label{eq:hcurve}
\Curve\ :\
y^2=(a_1-x)\dots(a_d-x)(\alpha_1-x)\dots(\alpha_{d-1}-x).
\end{equation}
Denote by $P_{b}(b,0)$, $P_{\infty}(\infty,\infty)$ its Weierstrass points,
$$
b\in\{a_1,\dots,a_d,\alpha_1,\dots,\alpha_d\}.
$$
For a divisor $D$ on the curve, we denote:
\begin{gather*}
\LL(D)=\left\{ f \text{ -- meromorphic function on } \Curve \mid (f)+D\ge 0 \right\},
 \\
\Omega(D)=\left\{ \omega \text{ -- meromorphic differential on } \Curve \mid (\omega)\ge D \right\}.
\end{gather*}
The Riemann-Roch theorem states that
$$
\dim\LL(D)=\deg D- g+\dim\Omega(D)+1,
$$
where $g$ is the genus of the curve.
In our case, $g=d-1$.

Now we are going to recall a Cayley-type condition for periodicity of billiard trajectories within an ellipsoid.

\begin{theorem}[\cites{DragRadn2004, DragRadn2011book}]\label{th:genCayley}
Consider the billiard within $\E$ and its trajectory with the caustics
$\Q_{\alpha_1}$, \dots, $\Q_{\alpha_{d-1}}$.
Denote
$$
\sqrt{\Pol(x)} = C_0 + C_1x + C_2 x^2 + C_3 x^3 + \dots.
$$

The trajectory is periodic with the elliptic period $m$ if and only if the following condition $C(m, d)$ is satisfied:
\begin{equation}\label{eq:even} \rank \left ( \begin{array}{llll}
     C_{d+1}     & C_{d+2}     & \dots & C_{m+1} \\
     C_{d+2}     & C_{d+3}     & \dots & C_{m+2} \\
      & & \dots                          \\
     C_{m+d-1} & C_{m+d} & \dots & C_{2m-1}
     \end{array} \right )<m-d+1.
     \end{equation}
\end{theorem}

Now we are going to reformulate above Cayley-type criterion in a form
of the Pad\'e approximation.

\begin{proposition}\label{prop:cayley-poly}
The condition $C(m, d)$ is satisfied if and only if there exist a pair of polynomials $p_m$ and $q_{m-d}$ of degree $m$ and $m-d$ respectively such that:
$$p_m(x)+q_{m-d}(x)\sqrt{\Pol(x)}=\mathcal O(x^{2m}).$$
\end{proposition}
\begin{proof}  Let
$$q_{m-d}(x)=\sum_0^{m-d}g_jx^j, \quad p_m(x)=\sum_0^{m}r_kx^k.$$

Compare the coefficients with the degrees of $x^j$. For $j=0, \dots, j=p$
we are getting the equations:
$$ g_0C_0+r_0=0$$
$$ g_0C_1+g_1C_0+r_1=0...$$
$$g_0C_m+\dots + g_{m-d}C_{d} +r_m=0.$$
From these equations, the coefficients $r_k, k=0,\dots, m$ could be determined.
Compare now the powers $x^{m+1}, \dots, x^{2m-1}$. The nonzero vector
$$(g_0, g_1, \dots, g_{m-d})$$
is orthogonal to the rows
$$(C_{m+1}, C_m, \dots, C_{d+1}), \dots ,(C_{2m-1}, C_{2m-2}, \dots, C_{m+d-1})$$
of the given matrix of the dimensions $(m-1, m-d+1)$. Thus, the rank of the matrix is less than $m-d+1$.
\end{proof}

Let us make one more step in the algebro-geometric analysis of the periodic trajectories.

\begin{theorem}\label{th:uslov}
Consider the billiard within $\E$ and its trajectory with caustics
$\Q_{\alpha_1}$, \dots, $\Q_{\alpha_{d-1}}$.
Denote
$$
\Set=\left\{ \{\alpha_i,\alpha_j\} \mid i\neq j \text{ and } \Q_{\alpha_i}, \Q_{\alpha_j} \text{ are of the same type}\right\}
$$

The trajectory is $n$-periodic if and only if the following conditions are satisfied:
\begin{itemize}
 \item if $n$ is odd, then one of the caustics is an ellipsoid;
 \item there is a subset $\Set'$ of $\Set$ such that:
\begin{equation}\label{eq:uslov}
n(P_0-P_{b_1})+\sum_{\{\beta,\beta'\}\in\Set'}  (P_{\beta}-P_{\beta'})\sim0
\end{equation}
is satisfied on the hyperelliptic curve \eqref{eq:hcurve}.
\end{itemize}
Moreover, there cannot be more than one subset $\Set'$ satisfying the condition \eqref{eq:uslov}.
\end{theorem}
\begin{proof}
Is obtained by the application of results from \cite{DragRadn2004}, using the fact that
$2P\sim2Q$ for any two Weierstrass points of a hyperelliptic curve.

Existence of two different subsets satisfying \eqref{eq:uslov}, would imply that $R_1+\dots+R_k\sim Q_1+\dots+Q_k$, for some Weierstrass points
$R_1$, \dots, $R_k$, $Q_1$, \dots, $Q_k$, all distinct among themselves.
Applying the Reimann-Roch theorem, we get that $\dim\LL(R_1+\dots+R_k)=1$.
Thus the space $\LL(R_1+\dots+R_k)$ consists only of constant functions, so the requested equivalence of divisors is not possible.
\end{proof}

Before we switch to the study of periodic trajectories with a low number of
impacts, let us review the properties of the winding numbers which are scattered
throughout this section, and which will be often used in the sequel.

\begin{lemma}\label{lemma:winding}
Let $(m_0, m_1, \dots, m_{d-1})$ be the winding numbers of a given periodic billiard trajectory.
Then:
\begin{itemize}
\item[(i)] the period $m_0$ is equal to the elliptic period if and only if the winding numbers are not all even;
\item[(ii)] if the winding number $m_j$, for $j>0$ is odd, then $b_{2j-2}$ and $b_{2j-1}$ are both in the set
$\{\alpha_1,\dots,\alpha_{d-1}\}$;
\item[(iii)] two consecutive winding numbers cannot both be odd.
\end{itemize}
\end{lemma}

\begin{proof}
The last item follows from (ii) and Lemma \ref{lemma:audin}.
\end{proof}

\section{Periodic trajectories with low periods}
\label{sec:low}

The study of billiard trajectories with a low number of impacts originated in \cite{DragRadn1998b}. There it was proven that $n$-periodic trajectories with period less than the dimension of the ambient space necessarily lie in a coordinate hyperplane. In the current paper we want to study trajectories with
the period $T$, for
$$d\le T\le 2d;$$
we will call them \emph{ the low impact trajectories}.

\subsection{$d$-periodic trajectories}

To start with, in the next statement, we improve a result from \cite{DragRadn1998b}.

\begin{theorem}\label{th:d}
 Each $n$-periodic trajectory within $\E$ with period $n\le d$ is contained in one of the coordinate hyperplanes.
\end{theorem}
\begin{proof}
The case $n<d$ is proved in \cite{DragRadn1998b}, so we need to consider only $n=d$.

Suppose first that $d$ is even.
The condition \eqref{eq:uslov} is equivalent to:
$$
dP_0+\sum P_{\beta}\sim dP_{\infty}+\sum P_{\beta'}.
$$
Notice that $\#\Set'\le\frac{d}{2}-1$, so $\dim\Omega(dP_{\infty}+\sum P_{\beta'})=\frac{d}{2}-1-\#\Set'$,
and, by Riemann-Roch theorem:
$$
\begin{aligned}
\dim\LL\left(dP_{\infty}+\sum P_{\beta'}\right)\ &=(d+\#\Set')-(d-1)+\left(\frac{d}{2}-1-\#\Set'\right)+1
\\
&=\frac{d}2+1.
\end{aligned}
$$
A basis of $\LL\left(dP_{\infty}+\sum P_{\beta'}\right)$ is $1,x,\dots,x^{d/2}$, so it cannot contain a function with the divisor of zeroes
$dP_0+\sum P_{\beta}$.

Now take odd $d$.
One of the caustics is ellipsoid, i.e.~ $b_1=\alpha_1$.
The condition \eqref{eq:uslov} is equivalent to:
$$
dP_0+\sum P_{\beta}\sim (d-1)P_{\infty}+P_{\alpha_1}+\sum P_{\beta'}.
$$
Notice that $\#\Set'\le\frac{d-1}{2}-1$, so $\dim\Omega((d-1)P_{\infty}+P_{\alpha_1}+\sum P_{\beta'})=\frac{d-1}{2}-1-\#\Set'$
and by the Riemann-Roch theorem:
$$
\begin{aligned}
\dim\LL&\left((d-1)P_{\infty}+P_{\alpha_1}+\sum P_{\beta'}\right)
\\
&\,=(d+\#\Set')-(d-1)+\left(\frac{d-1}{2}-1-\#\Set'\right)+1
\\
&\,=\frac{d-1}2+1.
\end{aligned}
$$
A basis of $\LL\left((d-1)P_{\infty}+P_{\alpha_1}+\sum P_{\beta'}\right)$ is $1,x,\dots,x^{(d-1)/2}$, so the conclusion is as in the previous case.

All of that implies that, whenever none of the values $\alpha_1$, \dots, $\alpha_{d-1}$ is in the set $\{a_1,\dots,a_d\}$, a $d$-periodic trajectory with caustics $\Q_{\alpha_1}$, \dots, $\Q_{\alpha_{d-1}}$ cannot exist.
On the other hand, if $\alpha_j=a_{j'}$ for some $j$, $j'$, the corresponding trajectories are either asymptotically approaching the coordinate hyper-plane $x_{j'}=0$ or are completely placed in that hyper-plane.
The former trajectories cannot be periodic, while the latter ones should be analysed as in the case of dimension $d-1$.
\end{proof}

\subsection{$(d+1)$-periodic trajectories}
Trajectories with low number of impacts of general billiards in the $d$-di\-men\-sional space were studied in \cites{BenciG1989,Bezdek2009,Irie2012}.
Each of these works, under certain conditions, shows the existence of the closed trajectories with at most $d+1$ bounces.
For ellipsoidal billiards, we proved in Theorem \ref{th:d} that the trajectories of period at most $d$ are contained in a coordinate hyperplane, thus they have at least one degenerate caustic and are essentially in a space of lower dimension.

In the next theorem, we prove that $(d+1)$-periodic trajectories of billiards within ellipsoids can exist only with a unique type of non-degenerate caustics.

\begin{theorem}\label{th:d+1}
Let $\T$ be a $(d+1)$-periodic trajectory of billiard within ellipsoid $\E$, such that it is not contained in any of the coordinate hyperplanes.
Then its caustics $\Q_{\alpha_1}$, \dots, $\Q_{\alpha_{d-1}}$ satisfy:
\begin{itemize}
 \item if $d$ is even, then $\alpha_1\in(0,a_1)$ and $\alpha_j,\alpha_{j+1}\in(a_j,a_{j+1})$ for all $j\in\{2,4,\dots,d-2\}$;
 \item if $d$ is odd, then $\alpha_j,\alpha_{j+1}\in(a_j,a_{j+1})$ for all $j\in\{1,3,\dots,d-2\}$.
\end{itemize}
Moreover, each $(d+1)$-periodic trajectory touches each of its caustics odd number of times.
\end{theorem}
\begin{proof}
Suppose first $d$ is even.
The condition \eqref{eq:uslov} is equivalent to:
$$
(d+1)P_0+\sum P_{\beta}\sim dP_{\infty}+P_{\alpha_1}+\sum P_{\beta'}.
$$
Notice that $\#\Set'\le\frac{d}{2}-1$, so
$$
\dim\Omega(dP_{\infty}+P_{\alpha_1}+\sum P_{\beta'})=
\begin{cases}
\frac{d}{2}-2-\#\Set' &\text{if}\quad \#\Set'\le\frac{d}{2}-2;
\\
0&\text{if}\quad \#\Set'=\frac{d}{2}-1.
\end{cases}
$$
The Riemann-Roch theorem yields:
$$
\dim\LL\left(dP_{\infty}+P_{\alpha_1}+\sum P_{\beta'}\right)=
\begin{cases}
 \frac{d}2+1 &\text{if}\quad \#\Set'\le\frac{d}{2}-2;
 \\
 \frac{d}2+2 &\text{if}\quad \#\Set'=\frac{d}{2}-1.
\end{cases}
$$
So, a basis for $\LL\left(dP_{\infty}+P_{\alpha_1}+\sum P_{\beta'}\right)$ is:
\begin{gather*}
1,x,\dots,x^{d/2} \quad\text{for}\quad\#\Set'\le\frac{d}{2}-2;
\\
1,x,\dots,x^{d/2},\frac{y}{(\alpha_1-x){\Pi_{\beta'}(\beta'-x)}} \quad\text{for}\quad\#\Set'=\frac{d}{2}-1.
\end{gather*}

Now take odd $d$.
The condition \eqref{eq:uslov} is equivalent to:
$$
(d+1)P_0+\sum P_{\beta}\sim (d+1)P_{\infty}+\sum P_{\beta'}.
$$
Notice that $\#\Set'\le\frac{d-1}{2}$, so
$$
\dim\Omega((d+1)P_{\infty}+\sum P_{\beta'})=
\begin{cases}
\frac{d-1}{2}-1-\#\Set' &\text{if}\quad \#\Set'\le\frac{d-1}{2}-1;
\\
0&\text{if}\quad \#\Set'=\frac{d-1}{2},
\end{cases}
$$
and by the Riemann-Roch theorem:
$$
\dim\LL\left((d+1)P_{\infty}+\sum P_{\beta'}\right)=
\begin{cases}
 \frac{d+1}2+1 &\text{if}\quad \#\Set'\le\frac{d-1}{2}-1;
 \\
 \frac{d+1}2+2 &\text{if}\quad \#\Set'=\frac{d-1}{2}.
\end{cases}
$$
Thus a basis for $\LL\left((d+1)P_{\infty}+\sum P_{\beta'}\right)$ is:
\begin{gather*}
1,x,\dots,x^{(d+1)/2} \quad\text{for}\quad\#\Set'\le\frac{d-1}{2}-1;
\\
1,x,\dots,x^{(d+1)/2},\frac{y}{\Pi_{\beta'}(\beta'-x)} \quad\text{for}\quad\#\Set'=\frac{d-1}{2}-1.
\end{gather*}

As in the proof of Theorem \ref{th:d}, we conclude that the bases consisting only of the powers of $x$ cannot give the desired equivalence relation.
Thus, a $(d+1)$-periodic trajectory has $[\frac{d-1}{2}]$ pairs of caustics of the same type and, if $d+1$ is odd, an ellipsoid as a caustic.
\end{proof}

Now we are ready to formulate Cayley-type conditions for $(d+1)$-periodic trajectories:
\begin{theorem}\label{th:cayley}
There exists a $(d+1)$-periodic trajectory of billiard within ellipsoid $\E$ with non-degenerate caustics
$\Q_{\alpha_1}$, \dots, $\Q_{\alpha_{d-1}}$ if and only if
\begin{itemize}
 \item for even $d$:
 \begin{itemize}
\item $\alpha_1\in(0,a_1)$ and $\alpha_j,\alpha_{j+1}\in(a_j,a_{j+1})$ for all $s\in\{2,4,\dots,d-2\}$;
\item $C_{d/2+1}=\dots=C_{d}=0$; and
\item
$C_0+C_1\alpha_j+\dots+C_{d/2}\alpha_j^{d/2}=0$ for all $j\in\{2,4,\dots,d-2\}$,
\end{itemize}
with
$$
\frac{\sqrt{\Pol(x)}}{(\alpha_1-x)(\alpha_3-x)\dots(\alpha_{d-1}-x)}=C_0+C_1x+C_2x^2+\dots;
$$

  \item for odd $d$:
 \begin{itemize}
\item $\alpha_j,\alpha_{j+1}\in(a_j,a_{j+1})$ for all $j\in\{1,3,\dots,d-2\}$; and
\item $C_{(d+1)/2+1}=\dots=C_{d}=0$; and
\item
$C_0+C_1\alpha_j+\dots+C_{(d+1)/2}\alpha_j^{(d+1)/2}=0$ for all $j\in\{1,3,\dots,d-2\}$,
\end{itemize}
with
$$
\frac{\sqrt{\Pol(x)}}{(\alpha_2-x)(\alpha_4-x)\dots(\alpha_{d-1}-x)}=C_0+C_1x+C_2x^2+\dots.
$$
\end{itemize}
\end{theorem}

\begin{proof}
The proof in both cases is similar, thus we give it only for even $d$.

According to Theorems \ref{th:d+1} and \ref{th:uslov}, the existence of a $(d+1)$-periodic trajectory with caustics $\Q_{\alpha_1}$, \dots, $\Q_{\alpha_{d-1}}$
is equivalent to the first of the listed relations and
$$
(d+1)P_0+P_{\alpha_2}+P_{\alpha_4}+\dots+P_{\alpha_d}\sim dP_{\infty}+P_{\alpha_1}+P_{\alpha_3}+\dots+P_{\alpha_{d-1}}.
$$
These two divisors are equivalent on $\Curve$ if and only if there is a meromorphic function $\varphi$, such that
$$
(\varphi)\sim(d+1)P_0+P_{\alpha_2}+P_{\alpha_4}+\dots+P_{\alpha_{d-2}} - \left(dP_{\infty}+P_{\alpha_1}+P_{\alpha_3}+\dots+P_{\alpha_{d-1}}\right).
$$
In other words, $\varphi$ is in $\LL\left(dP_{\infty}+P_{\alpha_1}+P_{\alpha_3}+\dots+P_{\alpha_{d-1}}\right)$ and has a zero of order $d+1$ at $P_0$ and simple zeroes at $P_{\alpha_2}$, $P_{\alpha_4}$, \dots, $P_{\alpha_{d-2}}$.
According to the proof of Theorem \ref{th:d+1}, the basis of $\LL\left(dP_{\infty}+P_{\alpha_1}+P_{\alpha_3}+\dots+P_{\alpha_{d-1}}\right)$ is:
$$
1,\ x,\ x^2,\ \dots,\ x^{d/2},\ \frac{y}{(\alpha_1-x)(\alpha_3-x)\dots(\alpha_{d-1}-x)},
$$
so we may search for $\varphi$ in the form:
\begin{equation}\label{eq:phi}
\varphi=A_0+A_1x+\dots+A_{d/2}x^{d/2}-\frac{y}{(\alpha_1-x)(\alpha_3-x)\dots(\alpha_{d-1}-x)}.
\end{equation}
$P_0$ is its zero of order $d+1$ if and only if $C_{d/2+1}=\dots=C_{d}=0$ and we set
$$
A_0=C_0,\ \dots,\ A_{d/2}=C_{d/2}.
$$
Then, since $y$ has zeroes at $P_{\alpha_2}$, $P_{\alpha_4}$, \dots, $P_{\alpha_{d-2}}$, it is needed only that $C_0+C_1x+\dots+C_{d/2}x^{d/2}$ has roots  $\alpha_2$, $\alpha_4$, \dots, $\alpha_{d-2}$.
\end{proof}

The conditions obtained in Theorem \ref{th:cayley} can be written in a more symmetric form:

\begin{theorem}\label{th:cayley-sym}
 There exists a $(d+1)$-periodic trajectory of billiard within ellipsoid $\E$ with non-degenerate caustics
$\Q_{\alpha_1}$, \dots, $\Q_{\alpha_{d-1}}$ if and only if
\begin{itemize}
 \item $\alpha_1\in(0,a_1)$ and $\alpha_j,\alpha_{j+1}\in(a_j,a_{j+1})$ for all $j\in\{2,4,\dots,d-2\}$ if $d$ is even;
 \item $\alpha_j,\alpha_{j+1}\in(a_j,a_{j+1})$ for all $j\in\{1,3,\dots,d-2\}$ if $d$ is odd,
\end{itemize}
and $\alpha_1$, \dots, $\alpha_{d-1}$ are roots of the polynomial $B_0+B_1x+\dots+B_d x^d$, where
$$
\sqrt{\Pol(x)}=B_0+B_1x+B_2x^2+\dots.
$$

\end{theorem}
\begin{proof}
Obtained by multiplying the expression \eqref{eq:phi} by the denominator of the last term on the righthand side.
\end{proof}

\begin{lemma}\label{lemma:polinomi}
Let $\rho_{2d}$ be a polynomial of degree $2d$ with no multiple roots.
Then the following statements are equivalent:
\begin{itemize}
\item[(a)]
there are a first degree polynomial $p_1(s)$, a constant $c$, and a factorization $\rho_{2d}=\rho_{d-1}\rho_{d+1}$ into polynomials of degrees $d-1$ and $d+1$ such that $p_1^2\rho_{d-1}-\rho_{d+1}=c$;
\item[(b)]
there are polynomials $\hat{p}_{d+1}$, $\hat{q}_1$ of degrees $d+1$ and $1$ and a constant $c'$ such that
$\hat{p}_{d+1}^2-\hat{q}_1^2\rho_{2d}=c'$.
\end{itemize}
\end{lemma}
\begin{proof}
Suppose (a) is satisfied.
 Denote:
 $$
 \hat{p}_{d+1}:=p_1^2\rho_{d-1}-\frac{c}2.
 $$
 We get:
 $$
  \hat{p}_{d+1}^2
 =
 p_1^2\rho_{d-1}(p_1^2\rho_{d-1}-c)+\frac{c^2}{4}
 =
p_1^2\rho_{d-1}\rho_{d+1}+\frac{c^2}{4}
=
p_1^2\rho_{2d}+\frac{c^2}{4},
  $$
so we get (b) with $\hat q_{1}=p_1$ and $c'=\frac{c^2}{4}$.

Now, suppose (b) is true.
Then:
$$
\left(\hat{p}_{d+1}-\sqrt{c'}\right)\left(\hat{p}_{d+1}+\sqrt{c'}\right)=\hat{q}_1^2\rho_{2d},
$$
which implies
$$
\hat{p}_{d+1}-\sqrt{c'}=\rho_{d+1},
\quad
\hat{p}_{d+1}+\sqrt{c'}=\hat{q}_1^2\rho_{d-1},
$$
for some polynomials $\rho_{d+1}$, $\rho_{d-1}$ of degrees $d+1$, $d-1$, such that $\rho_{2d}=\rho_{d-1}\rho_{d+1}$.
Subtracting the obtained relations, we get:
$$
2\sqrt{c'}=\hat{q}_1^2\rho_{d-1}-\rho_{d+1},
$$
so (a) follows, with $c=2\sqrt{c'}$ and $p_1=\hat{q}_1$.
\end{proof}

\begin{corollary}
 If there exists a $(d+1)$-periodic billiard trajectory within ellipsoid $\E$ with non-degenerate caustics $\Q_{\alpha_1}$, \dots, $\Q_{\alpha_{d-1}}$, then two following functional relations are satisfied:
 \begin{itemize}
\item[\textbf{(a)}]
  There exist a first degree polynomial $p_1(s)$ and a constant $c$ such that
  \begin{equation}\label{eq:p1}
  p_1^2(s)\prod_{j=1}^{d-1}\left(s-\frac1{\alpha_j}\right)-s\prod_{j=1}^d\left(s-\frac1{a_j}\right)=c.
  \end{equation}
\item[\textbf{(b)}]
  There exist polynomials $\hat{p}_{d+1}(s)$, $\hat{q}_1(s)$ of degrees $d+1$, $1$ respectively, and a constant $c'$ such that
  \begin{equation}\label{eq:q1}
   \hat{p}_{d+1}^2(s)-\hat{q}_1^2(s)s\prod_{j=1}^d\left(s-\frac1{a_j}\right)\prod_{j=1}^{d-1}\left(s-\frac1{\alpha_j}\right)=c'.
  \end{equation}
\end{itemize}
Moreover, the following relationship holds:
$$
\hat{p}_{d+1}(s)=p_1^2(s)\prod_{j=1}^{d-1}\left(s-\frac1{\alpha_j}\right)-\frac{c}{2},
\quad
\hat{q}_1(s)=p_1(s),
\quad
c'=\frac{c^2}{4}.
$$
\end{corollary}
\begin{proof}
 We will present the proof for $d$ odd only, since the case when $d$ is even is similar.
 The second and third conditions of Theorem \ref{th:cayley} together are equivalent to the existence of a polynomial $r_1(x)$ of degree $1$ such that the expression:
 \begin{equation}\label{eq:r1}
  r_1(x)(\alpha_1-x)(\alpha_3-x)\dots(\alpha_{d-2}-x) - \frac{y}{(\alpha_{2}-x)(\alpha_4-x)\dots(\alpha_{d-1}-x)}
 \end{equation}
has zeroes at points $P_{\alpha_j}$ for $j\in\{1,3,\dots,d-2\}$ and a zero of order $d+1$ for $x=0$.
Thus,
\begin{equation}\label{eq:r1-y}
r_1(x)-\frac{y}{\prod_{j=1}^{d-1}(\alpha_j-x)}
\end{equation}
has a zero of order $d+1$ for $x=0$.
Now, by multiplying \eqref{eq:r1-y} by
$$
  \left(r_1(x)+\frac{y}{\prod_{j=1}^{d-1}(\alpha_j-x)}\right)\prod_{j=1}^{d-1}(\alpha_j-x),
  $$
we get that
$$
r_1^2(x)\prod_{j=1}^{d-1}(\alpha_j-x)-\prod_{j=1}^{d}(a_j-x)=\tilde{c} x^{d+1},
$$
for some constant $\tilde{c}$, since its lefthand side is polynomial in $x$ of degree $d+1$ and has a zero of order $d+1$ for $x=0$.

 Now, dividing by $x^{d+1}$ and introducing $s=1/x$, one gets that Theorem \ref{th:cayley} implies (a).
Then (b) is true according to Lemma \ref{lemma:polinomi}.
 \end{proof}

\subsection{$(d+k)$-periodic trajectories}
In this section, we will show that periodic trajectories with $d+k$ bounces, $k\in\{1,\dots,d-2\}$ always have pairs of caustics of the same type.

\begin{theorem}\label{th:d+k}
Any $(d+k)$-periodic trajectory of billiard within ellipsoid $\E$ has at least $[\frac{d-k}2]$ pairs of caustics of the same type, with $k\in\{1,\dots,d-2\}$.
\end{theorem}
\begin{proof}
Let the caustics of a given $(d+k)$-periodic trajectory be $\Q_{\alpha_1}$, \dots, $\Q_{\alpha_{d-1}}$, and take $\Set'$ as in Theorem \ref{th:uslov}.
According to that theorem, the condition \eqref{eq:uslov} is satisfied:
\begin{equation}\label{eq:d+k}
(d+k)(P_0-P_{b_1})+\sum_{\{\beta,\beta'\}\in\Set'}(P_{\beta}-P_{\beta'})\sim0.
\end{equation}

Suppose first that $d+k$ is even.
The condition \eqref{eq:d+k} is then equivalent to:
$$
(d+k)P_0+\sum P_{\beta}\sim (d+k)P_{\infty}+\sum P_{\beta'}.
$$
Since
$\dim\Omega\left((d+k)P_{\infty}\right)=d-1-\frac{d+k}{2}$,
we have:
$$
\dim\Omega\left((d+k)P_{\infty}+\sum P_{\beta'}\right)
=
\max\left\{0,d-1-\frac{d+k}{2}-\#\Set'\right\}.
$$
Notice that $d-1-\frac{d+k}{2}-\#\Set'\ge0$ whenever $\#\Set'<\frac{d-k}{2}$.
In that case, the Riemann-Roch theorem gives:
$$
\begin{aligned}
\dim\LL&\left((d+k)P_{\infty}+\sum P_{\beta'}\right)=
\\
&=
(d+k+\#\Set')-(d-1)+ \left(d-1-\frac{d+k}{2}-\#\Set'\right)+1
\\
&=\frac{d+k}2+1.
\end{aligned}
$$
Thus, $1,x,\dots,x^{(d+k)/2}$ is a basis of $\LL\left((d+k)P_{\infty}+\sum P_{\beta'}\right)$, so that space cannot contain any functions with the zeroes divisor
$(d+k)P_0+\sum P_{\beta}$.

Now suppose $d+k$ is odd.
One of the caustics then must be an ellipsoid.
The condition \eqref{eq:d+k} is then equivalent to:
$$
(d+k)P_0+\sum P_{\beta}\sim (d+k-1)P_{\infty}+P_{b_1}+\sum P_{\beta'}.
$$
Since
$\dim\Omega\left((d+k-1)P_{\infty}\right)=d-1-\frac{d+k-1}{2}$,
we have:
$$
\begin{aligned}
\dim\Omega&\left((d+k-1)P_{\infty}+P_{b_1}+\sum P_{\beta'}\right)
=
\\
&=
\max\left\{0,d-1-\frac{d+k-1}{2}-\#\Set'-1\right\}.
\end{aligned}
$$
Notice that $d-1-\frac{d+k-1}{2}-\#\Set'-1\ge0$ whenever $\#\Set'<\frac{d-k-1}{2}=[\frac{d-k}2]$.
In that case, the Riemann-Roch theorem gives:
$$
\begin{aligned}
\dim\LL&\left((d+k-1)P_{\infty}+P_{b_1}+\sum P_{\beta'}\right)
\\ &=
(d+k+\#\Set')-(d-1)+ \left(d-1-\frac{d+k-1}{2}-\#\Set'-1\right)+1
\\
&=\frac{d+k-1}2+1.
\end{aligned}
$$
Thus, $1,x,\dots,x^{(d+k-1)/2}$ is a basis of $\LL\left((d+k-1)P_{\infty}+P_{b_1}+\sum P_{\beta'}\right)$, so that space cannot contain any functions with the zeroes divisor
$(d+k)P_0+\sum P_{\beta}$.
\end{proof}

\begin{remark}
We conclude that $\Set'$ has at least $[\frac{d-k}{2}]$ elements.
For $\#\Set'\ge[\frac{d-k}2]$, we get that the corresponding space of meromorphic functions contains functions which cannot be expressed as polynomials or rational functions in $x$.
In such case, we are able to express the analytic conditions for the existence of periodic trajectories.
Some examples are presented in Section \ref{sec:examples}.
\end{remark}

\section{Poncelet polygons and extremal polynomials}
\label{sec:poly}

In \cite{DragRadn1998b}, the generalized Cayley's condition $C(n,d)$ has been formulated as a rank condition on a rectangular matrix, to describe $n$-periodic billiard trajectories within a given ellipsoid in the $d$-dimensional space, see Theorem \ref{th:genCayley} of this paper.
These conditions were rewritten in a polynomial form in \cite{RR2014}.
We recall that the original Cayley’s conditions which describe  the Poncelet polygons in plane were reformulated in polynomial/Pad\'e's approximation form by Halphen in \cite{Hal1888} (see Theorem \ref{th:halphen} above and also \cites{DragRadn2011book,Drag2009}).
In this section, we are going to prove the conjectures from \cite{RR2014} by recognizing  the role of the Pell equations and the extremal nature of the polynomials involved.
We fully exploit the classical and modern approximation theory, in particular the recent deep findings about the generalized Chebyshev polynomials,
(see \cites{AhiezerAPPROX, KLN1990, PS1999} and references therein).

We take constants $b_1$, \dots, $b_{2d-1}$ as defined in Section \ref{sec:hyper} and denote:
\begin{gather*}
c_1=\frac{1}{b_1},\ \dots,\ c_{2d-1}=\frac{1}{b_{2d-1}},\ c_{2d}=0,
\\
\hat{\mathcal{P}}_{2d}(s)=\prod_{j=1}^{2d}(s-c_j).
\end{gather*}

Let us rewrite the conditions from \cite{RR2014} in a different form.

\begin{proposition}
The generalized Cayley's condition $C(n,d)$ is satisfied if and only if there exist a pair of real polynomials
$\hat{p}_n$, $\hat{q}_{n-d}$
of degrees $n$ and $n-d$ respectively such that the Pell equation holds:
\begin{equation}\label{eq:pell}
\hat{p}_n^2(s)-\hat{\mathcal{P}}_{2d}(s)\hat{q}_{n-d}^2(s)=1.
\end{equation}
\end{proposition}

In \cite{DragRadn2004} yet another condition of $n$-periodic billiard trajectories was derived in the Abel-Jacobi form, based on the analysis of the elliptical coordinates:
there exist integers $m_j$, $j\in\{1,\dots,d-1\}$ such that:
$$
n\int_{c_1}^{\infty}\frac{s^k}{\sqrt{\hat{\mathcal{P}}_{2d}(s)}}ds
=
\sum_{j=2}^d (-1)^j m_{j-1}\int_{c_{2j-1}}^{c_{2j-2}}\frac{s^k}{\sqrt{\hat{\mathcal{P}}_{2d}(s)}}ds,
\quad
k\in\{0,1,\dots,d-2\}.
$$

We recall that, following \cite{RR2014}, we have denoted $m_0=n$, $m_d=0$, and have called
$m_0$, $m_1$, \dots, $m_{d-1}$
\emph{the winding numbers} of a given periodic billiard trajectory.
Similarly \cite{RR2014} denoted by $\tau_j$ the number of zeroes of $\hat{q}_{n-d}$ in the interval $(c_{2j},c_{2j-1})$.
The $d$-tuple $(\tau_1,\dots,\tau_d)$ is called \emph{the signature}.

The polynomials $\hat{p}_n$ are extremal polynomials on the system of $d$ intervals
$[c_{2d},c_{2d-1}]\cup[c_{2d-2},c_{2d-3}]\cup\dots\cup[c_2,c_1]$.
Following the principles formulated by Chebyshev and his school (see \cite{AhiezerAPPROX}), we are going to study the structure of extremal points of $\hat{p}_n$, in particular the set of points of alternance.

Notice that the roots of $\hat{\mathcal{P}}_{2d}(s)$ are simple solutions of the equation $\hat{p}_n^2(s)=1$, while the roots of
$\hat{q}_{n-d}(s)$ are double solutions of the equation $\hat{p}_n^2(s)=1$.
Because of the degrees of the polynomials, these are all points where $\hat{p}_n^2(s)$ equals to unity.

Let us recall that a set of \emph{points of alternance} is, by definition, a subset of the solutions of the equation $\hat{p}_n^2(s)=1$, with the maximal number of elements, such that the signs of $\hat{p}_n$ alter on it.
Such a set is not uniquely determined, however the number of its elements is fixed and equal to $n+1$.

\begin{theorem}\label{th:winding}
\begin{itemize}
 \item[\textbf{(a)}] The winding numbers satisfy:
 $$
 m_{j}=m_{j+1}+\tau_{j}+1,
 \quad
 1\le j\le d.
 $$
 \item[\textbf{(b)}] The winding numbers are strictly decreasing:
 $$
 m_{d-1}<m_{d-2}<\dots<m_1<m_0.
 $$
\end{itemize}
\end{theorem}
\begin{proof}
It follows from \cite{KLN1990} (see also \cite{SodinYu1992}) that the number of points of alternance of the polynomial $\hat{p}_n$ on the segment $[c_{2d},c_{2j+1}]$ is equal to $1+m_j$, for
$j\in\{0,\dots,d-1\}$.
The difference $m_{j-1}-m_j$ is thus equal to the number of points of alternance on the interval $[c_{2j+1},c_{2j-1}]$.
According to the structure of the sets of the alternance, that number equals the sum of the numbers of the double points of alternance from the interval $(c_{2j},c_{2j-1})$ and one simple point of alternance at one of the endpoints of the interval.
This proves (a) and (b) follows immediately.
\end{proof}

\begin{corollary}\label{cor:Q}
 All zeroes of $\hat{q}_{n-d}$ are real.
\end{corollary}
\begin{proof}
 The polynomial $\hat{p}_n$ has $n-d$ double extremal points in the interior of the union of the intervals
$(c_{2d}, c_{2d-1}) \cup\dots \cup (c_2, c_1)$. These roots of $\hat{p}_n'$ coincide with the roots of the polynomial $\hat{q}_{n-d}$ of degree $n-d$.
\end{proof}

\begin{remark}\label{rem:ramirez}
Theorem \ref{th:winding} answers affirmatively to Conjectures 1 and 3 from \cite{RR2014}.
The Conjecture 2 from \cite{RR2014} is answered affirmatively in Corollary \ref{cor:Q}.
\end{remark}

\begin{example}\label{ex:d-elliptic}
 Consider the case of a trajectory with non-degenerate caustics with elliptic period $d$.
According to Theorem \ref{th:d}, such trajectory is periodic with period $2d$ in the Cartesian coordinates.
 \cite{RR2014}*{Theorem 13} conjectured the winding numbers of such trajectories, based on \cite{RR2014}*{Conjecture 1}.
 Now, having the Conjecture 1 proved, we certify the winding numbers to be $m_j=2(d-j)$, $j=0,\dots,d-1$.
\end{example}

Now, consider the case $n=d+1$.
Since, according to Theorem \ref{th:d+1}, among $m_i$, even and odd numbers alternate and decrease, and $m_0=d+1$, we get

\begin{proposition}\label{prop:unique}
 \begin{itemize}
  \item[\textbf{(a)}] The winding numbers of the trajectories of period $d+1$ within an ellipsoid in the $d$-dimensional space are
  $$
  (m_0,m_1,\dots, m_{d-1})=(d+1, d, d-1,\dots,3,2).
  $$
  \item[\textbf{(b)}] The signature of such trajectories is
  $(0,0,\dots,0,1)$.
 \end{itemize}
\end{proposition}

\begin{theorem}
 For a given ellipsoid from a confocal pencil in the Euclidean space $\mathbf{E}^d$, the set of caustics which generates $(d+1)$-periodic trajectories is unique, if it exists.
\end{theorem}
\begin{proof}
Consider the Pell equation \eqref{eq:pell}.
The polynomial $\hat{p}_{d+1}$ has one ($1=d+1-d$) double point of the alternance, which is the zero of polynomial $\hat{q}_1$:
$\hat{q}_1(\gamma)=0$, with $\gamma\in(c_{2d}=0,c_{2d-1})$, according to Propositon \ref{prop:unique}.
In addition, $\hat{p}_{d+1}(s)$ has $d+1$ simple points of the alternance at the endpoints of the intervals
$[c_{2d},c_{2d-1}]$,\dots,$[c_2,c_1]$.

The following properties of the polynomial $\hat{p}_{d+1}$ follows from the structure and distribution of the points of the alternance, the winding numbers and the signature:
\begin{itemize}
 \item $\hat{p}_{d+1}$ takes value $-1$ at $0$, $1/a_1$, \dots, $1/a_d$;
 \item in $(0,1/a_d)$, $\hat{p}_{d+1}$ has a local maximum equal to unity;
 \item $\hat{p}_{d+1}$ takes value $1$ at $1/\alpha_1$, \dots, $1/\alpha_{d-1}$.
\end{itemize}
See Figures \ref{fig:P4} and \ref{fig:P5} for graphs of $\hat{p}_{4}$ and $\hat{p}_{5}$.

\begin{figure}[h]
\centering
\includegraphics[width=9cm,height=4.2cm]{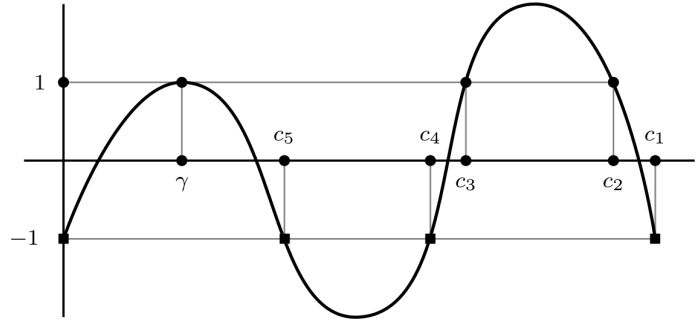}
\caption{The graph of $\hat{p}_4(s)$. The parameters are: $c_1=1/a_1$, $c_2=1/\alpha_1$, $c_3=1/\alpha_2$, $c_4=1/a_2$, $c_5=1/a_3$.}\label{fig:P4}
\end{figure}

\begin{figure}[h]
\centering
\includegraphics[width=11cm,height=4.17cm]{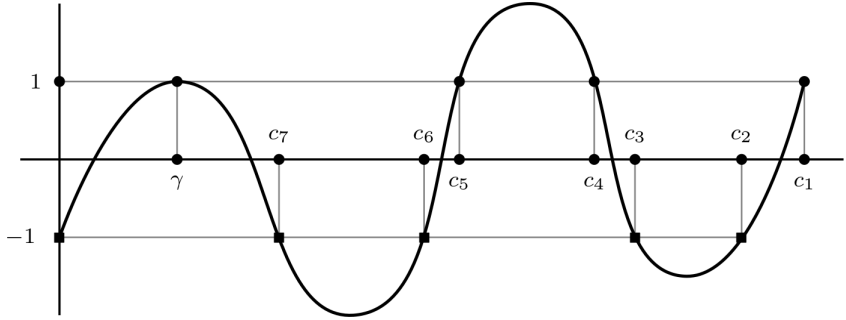}
\caption{The graph of $\hat{p}_5(s)$. The parameters are $c_{7}=1/a_4$, $c_{6}=1/a_{3}$, $c_{3}=1/a_{2}$, $c_{2}=1/a_{1}$; $c_5=1/\alpha_3$, $c_4=1/\alpha_2$, $c_1=1/\alpha_1$. The signature of the trajectory is $(0,0,0,1)$ and the winding numbers $(5,4,3,2)$. }\label{fig:P5}
\end{figure}
For each $d$, there is a unique polynomial satisfying the listed properties and it can be determined as follows.
Denote by $\gamma$ the only point of local extremum of
$$
r_{d+1}(s)=s\prod_{j=1}^{d}\left(s-\frac1{a_j}\right)
$$
in $(0,1/a_d)$ and define:
$$
\hat{p}_{d+1}(s)=\frac{2r_{d+1}(s)}{r_{d+1}(\gamma)}-1,
\quad
\hat{q}_1(s)=s-\gamma.
$$
Now, $1/\alpha_1$, \dots, $1/\alpha_{d-1}$ are solutions of the equation $\hat{p}_{d+1}=1$, different from $\gamma$.
If those solutions exist, they are uniquely determined.
\end{proof}

\begin{remark}
We note that the existence of $(d+1)$-periodic trajectories with non-degenerate caustics will depend on the shape of the ellipsoid $\E$.
The discussion about that in the $3$-dimensional case can be found in \cite{CRR2011}.
\end{remark}

\begin{theorem}
 For a given ellipsoid $\E$, the quadrics $\Q_{\lambda_k}$ are the caustics of $(d+1)$-periodic trajectories if $\lambda_1^{-1}$, \dots, $\lambda_{d-1}^{-1}$ are the solutions of the equation $\hat{p}_{d+1}(s)-1=0$, distinct from $\gamma$.
\end{theorem}

\section{Properties of the frequency map}
\label{sec:freq}
In this section, we will prove the injectivity property of the frequency map.

Let us start with two lemmas.

\begin{lemma}\label{lemma:same-type-caustics}
If $\Q_{\alpha}$, $\Q_{\beta}$ are caustics of the same type of a given billiard trajectory within $\E$, then
$\{\alpha^{-1}, \beta^{-1}\}=\{c_{2k+1}, c_{2k}\}$, for some $k$.
\end{lemma}
\begin{proof}
According to Lemma \ref{lemma:audin}, exactly one of each pair $\{b_{2i-1},b_{2i}\}$ is a parameter of a caustic of the trajectory.
Since $\Q_{\alpha}$, $\Q_{\beta}$ are of the same type, $\alpha$ and $\beta$ must be consecutive in the sequence $b_1,\dots,b_{2n-1}$, so
$\{\alpha,\beta\}=\{b_{2k},b_{2k+1}\}$ for some $k$.
\end{proof}

\begin{lemma}[Theorem 2.12 from \cite{PS1999}]\label{lemma:pell-unique}
Let $p_n$, $p_n^*$ be two polynomials of degree $n$, which solve the Pell's equations.
Denote by
$$
\I_d=\cup_{j=0}^{d-1}[c_{2(d-j)}, c_{2(d-j)-1}]
\quad\text{and}\quad
\I_d^*=\cup_{j=0}^{d-1}[c_{2(d-j)}^*, c_{2(d-j)-1}^*]
$$
respectively the sets $\{x\mid |p_n(x)|\le 1\}$ and $\{x\mid |p_n^*(x)|\le 1\}$.
Suppose that:
\begin{itemize} \item [\textrm{(i)}] at least one of the intervals from $\I_d$ coincides with one of the intervals from $\I_d^*$;
\item[\textrm{(ii)}] for each $j\in \{0, \dots, d-1\}$, $c_{2(d-j)}= c_{2(d-j)}^*$ or $c_{2(d-j)-1}= c_{2(d-j)-1}^*$;
\item[\textrm{(iii)}] in each pair of the corresponding intervals
$$
[c_{2(d-j)},c_{2(d-j)-1}]
\quad\text{and}\quad
[c_{2(d-j)}^*,c_{2(d-j)-1}^*]
$$
the polynomials  $p_n$, $p_n^*$  have the same number of extreme points.
\end{itemize}
Then the polynomials $p_n$, $p_n^*$  coincide up to a constant multiplier and sets $\I_d$ and $\I_d^*$ coincide.
\end{lemma}

\begin{theorem}\label{th:signature-caustics}
Given an ellipsoid $\E$ in $d$-dimensional space and $n>d$ an integer.
There is at most one set of caustics $\{\alpha_1, \dots, \alpha_{d-1}\}$ of the given types, which generates $n$-periodic trajectories  within $\E$ having a prescribed signature.
\end{theorem}

\begin{proof}
The assumption (i) of Lemma \ref{lemma:pell-unique} is satisfied since $[c_{2d}, c_{2d-1}]=[0, a_d^{-1}]$, and
the assumption (ii) due to the Lemma \ref{lemma:same-type-caustics}.
The assumption (iii) follows from the fact that the signature is given, which completes the proof of this Theorem.
\end{proof}

 In order to extend the considerations about winding numbers to the cases of irrational frequencies and non-periodic trajectories, we  employ the potential theory and harmonic analysis, see \cite{Si2011}. We consider the differential of the third kind defined by the conditions:
\begin{equation}\label{eq:diff}
\int_{c_{2j+1}}^{c_{2j}}\frac{\eta_{d-1}(s)}{\sqrt{\hat{\mathcal{P}}_{2d}(s)}}\, ds =0, \, j=d-1, d-2, ..., 1,
\end{equation}
where $\eta_{d-1}$ is a monic polynomial of degree $d-1$.
 Then the equilibrium measure $\mu$, defined by:
$$
\mu([c_{2k}, c_{2k-1}])=\frac{1}{\pi}\int_{c_{2k}}^{c_{2k-1}}\frac{|\eta_{d-1}(s)|}{\sqrt{|\hat{\mathcal{P}}_{2d}(s)|}}\, ds,
$$
induces  a map $\bf m:\mathcal R^{2d-2}_<  \rightarrow  \mathcal R^{d-1}_+$:
$$
\bf m: (c_{2d-2}, \dots, c_1) \mapsto (\mu([c_{2d-2}, c_{2d-3}]),...,\mu([c_{2}, c_{1}])),
$$
where $\mathcal R^{2d-2}_< $ denotes the finite increasing sequences of $2d-2$ of real numbers.
We use the considerations parallel to the proof of the Bogotaryev-Peherstorfer-Totik Theorem (Theorem 5.6.1 from \cite{Si2011}) and observe that they can be extended to all possible distributions of caustic parameters $\alpha$'s vs. the confocal family parameters $a$'s  as governed by  Lemma \ref{lemma:audin}.
Assuming that $c$'s reciprocal to $a$'s are fixed, and $\hat c$'s reciprocal to $\alpha$'s ($\hat c_i=\alpha_i^{-1}$) vary one proves that the above map understood now as a function of varying $\hat c$'s $\bf m (\hat c_{d-1}, \dots, \hat c_1)$ is everywhere locally injective: by using the properties of the polynomials $\eta_{d-1}$ and $\hat{\mathcal{P}}_{2d}$  and their derivatives with respect to the variable $\hat c$s as well as the derivative of the quotient
$$
\frac{|\eta_{d-1}(s)|}{\sqrt{\hat{\mathcal{P}}_{2d}(s)}},
$$

one gets that the Jacobian of the equilibrium measure map is diagonally dominant, thus invertible. From there we get the local injectivity of the frequency map,
defined by (see also Theorem 5.12.12 from \cite{Si2011}):
$$
F(\alpha_1, \dots, \alpha_{d-1})=(\mu([c_{2d}, c_{2d-1}]),\sum_{k=d-1}^{d}\mu([c_{2k}, c_{2k-1}]),...,\sum_{k=1}^d\mu([c_{2k}, c_{2k-1}])).
$$
Notice that a connected component of the set of the parameters of non-degenerate caustics consists of all possible parameters of certain given types of the caustics.
Thus Theorem \ref{th:signature-caustics} implies that the frequency map never attains one rational value twice on such a connected component.
This, together with the local injectivity, leads to the global injectivity of the frequency map.
\begin{theorem}\label{th:1-1a}
The frequency map for the billiard within ellipsoid is injective on each connected component of the set of the parameters of non-degenerate caustics.
\end{theorem}

\begin{remark}
Using \cite{Ap1986} one can show that in the case of periodic trajectories
the frequency map coincides with the one defined through the winding numbers and the numbers of points of alternance. Alternatively, by using bilinear relations between the differentials of the first and third kind, see for example \cite{Sp1957}, one can easily get the following relations:
\begin{equation}\label{eq:frequency}
\sum_{i=1}^{d-1}y_i\int_{\bf {a_i}}\omega_j=2y_d\int_{c_1}^\infty \omega_j,
\end{equation}
where the cycle $\bf {a_i}$ encircles the gap $[c_{2(d+1-i)-1}, c_{2(d-i)}]$ clockwise while, cycles
$\bf {b_i}$  are going around $[c_{2d}, c_{2(d+1-i)-1}]$ clockwise. The differentials $\omega_j$ form a basis of holomorphic differential on the curve
$$
\hat \Curve: t^2=\hat{\mathcal{P}}_{2d}(s),
$$
with
$$
\omega_j=\frac{s^{j-1}}{\sqrt{\hat{\mathcal{P}}_{2d}(s)}}ds, \quad j=1,\dots, d-1.
$$
If we denote the components of the map $F$ as $(f_1, \dots, f_d)$, then
$$
|y_i|=f_i.
$$
We  also observe the monotonicity  property of the frequency map: $f_1<f_2<\dots<f_d$.

\end{remark}

\begin{corollary}\label{cor:rotation2}
Given a confocal pencil of conics in the plane
\begin{equation}\label{eq:confocal2}
C_{\lambda}: \frac{x^2}{a-\lambda} + \frac {y^2}{b-\lambda}=1, \quad a>b>0.
\end{equation}
  Then the rotation number
\begin{equation}\label{eq:rotation}
\rho(\lambda)
=
\rho(\lambda, a, b)
=
\frac
{\int_0^{\min\{b, \lambda\}}\frac{dt}{\sqrt{(\lambda-t)(b-t)(a-t)}}}
{2\int^a_{\max\{b, \lambda\}}\frac{dt}{\sqrt{(\lambda-t)(b-t)(a-t)}}}
\end{equation}
is a  strictly monotonic function on each of the intervals $(-\infty, b)$ and $(b, a)$.
\end{corollary}

The relation (\ref{eq:rotation}) is equivalent to the $d=2$ case of (\ref{eq:frequency}).

\begin{remark}
Let us relate more closely the formula (\ref{eq:rotation}) with the geometric meaning of the rotation number.
Consider a billiard trajectory within $C_0$ with the caustic $C_{\alpha}$. 
We will assume that $C_{\alpha}$ is an ellipse, i.e.~$a>b>\alpha>0$.
Following \cite{DragRadn2014jmd}, consider the map:
$$
\sigma(A)=\left(
\int_{M_0}^A\frac{d\lambda_2}{\sqrt{(a-\lambda_2)(b-\lambda_2)(\alpha-\lambda_2)}},
\int_{M_0}^A\frac{d\lambda_1}{\sqrt{(a-\lambda_1)(b-\lambda_1)(\alpha-\lambda_1)}}
\right),
$$
where $M_0$ is an arbitrary given point on $C_{\alpha}$, and $\lambda_1$, $\lambda_2$ are Jacobi coordinates associated with the confocal pencil (\ref{eq:confocal2}).
According to \cite{DragRadn2014jmd}, $\sigma$ maps the region between $C_0$ and $C_{\alpha}$ bijectively to the cylinder $(\mathbf{R}/p\mathbf{Z})\times[-v,0]$, with
$$
p=4\int_b^a\frac{d\lambda}{\sqrt{(a-\lambda)(b-\lambda)(\alpha-\lambda)}},
\quad
v=\int_0^{\alpha}\frac{d\lambda}{\sqrt{(a-\lambda)(b-\lambda)(\alpha-\lambda)}}.
$$ 
The cylinder can be seen also as a rectangle with the horizontal sides equal to $p$, while the vertical ones are equal to $v$ and identified to each other.
The transformation $\sigma$ maps the billiard trajectory onto a zig-zag line, such that its segments form angles $\pm\pi/4$ with the sides of the rectangle, see Figure \ref{fig:zigzag}.
\begin{figure}[h]
	\begin{center}
	\begin{tikzpicture}

\draw [black, fill=gray!30] (0,0) circle [x radius=sqrt(7), y radius=2];
	
\draw[black,fill=white] (0,0) circle [x radius=2, y radius=1];

\draw[thick]
(-1.73205, 0.5) -- (-2.62148, -0.27027) -- (-1.1952, -0.801794) -- (1.31279, -1.73643) --
(1.95711, -0.205987);

\draw[black,fill=black] (-1.73205, 0.5) circle(0.05) node[below right]{$T_0$};
\draw[black,fill=black] (-2.62148, -0.27027) circle(0.05) node[left]{$A_1$};
\draw[black,fill=black] (-1.1952, -0.801794) circle(0.05) node[above right]{$T_1$};
\draw[black,fill=black] (1.31279, -1.73643) circle(0.05) node[below right]{$A_2$};
\draw[black,fill=black] (1.95711, -0.205987) circle(0.05) node[left]{$T_2$};

\begin{scope}[shift={(0,-4)}]

\draw[gray!30,fill=gray!30](-5,1)--(5,1)--(5,-0.5)--(-5,-0.5);

\draw(-5,1)--(5,1);
\draw(-5,-0.5)--(5,-0.5);
\draw[dashed](-5,1)--(-5,-0.5);
\draw[dashed](5,1)--(5,-0.5);

\draw[thick] (-4,-0.5)--(-2.5,1)--(-1,-0.5)--(0.5,1)--(2,-0.5);

\draw[black,fill=black] (-4, -0.5) circle(0.05) node[below]{$\sigma(T_0)$};
\draw[black,fill=black] (-1, -0.5) circle(0.05) node[below]{$\sigma(T_1)$};
\draw[black,fill=black] (2, -0.5) circle(0.05) node[below]{$\sigma(T_2)$};

\draw[black,fill=black] (-2.5, 1) circle(0.05) node[above]{$\sigma(A_1)$};
\draw[black,fill=black] (.5, 1) circle(0.05) node[above]{$\sigma(A_2)$};

\end{scope}

	\end{tikzpicture}
	
\end{center}
\caption{A billiard trajectory and its image by $\sigma$.}\label{fig:zigzag}
\end{figure}
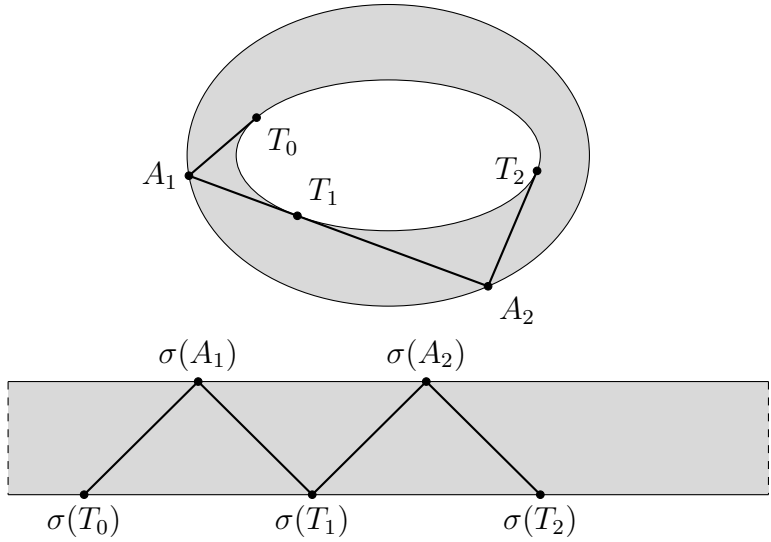

The pullback of the Lebesgue measure $dL$ on the horizontal side of the rectangle $\mu_0=\sigma^{*}(dL)$ is a measure on the caustic $C_{\alpha}$, which is invariant with respect to the billiard dynamics: $\mu_0(T_0T_1)=\mu_0(T_1T_2)$.
The rotation number is:
$$
\rho(\alpha,a,b)=\frac{\mu_0(T_1T_2)}{\mu_0(C_{\alpha})}=\frac{|\sigma(T_1)\sigma(T_2)|}{p}=\frac{2v}p,
$$ 
which is equivalent to (\ref{eq:rotation}).

\end{remark}

\begin{remark}
A proof of Corollary \ref{cor:rotation2} is contained in the beautiful book of Duistermaat \cite{DuistermaatBOOK}.
However, that proof is highly nontrivial and uses a heavy machinery of the theory of algebraic surfaces.
Let us mention that  statements similar to \cite{PS1999}*{Theorem 2.12} existed before, see for example \cite{Meiman1977}.
\end{remark}

\section{Trajectories with low periods in dimension three}
\label{sec:examples}
This section is meant to illustrate the power and effectiveness of the methods and tools developed above.
We provide a comprehensive description of periodic trajectories with periods $4$, $5$, and $6$ in the three-dimensional space.
By analysing these cases, we observed new, interesting properties of such periodic trajectories. For the two-dimensional case see \cite{DragRadn2019rcd}.

\subsection{4-periodic trajectories in dimension three}\label{sec:4periodic}

The Cayley type conditions for such trajectories can be written directly applying Theorem \ref{th:d+1}.

\begin{example}\label{ex:n=4}
There is a $4$-periodic trajectory of the billiard within ellipsoid $\E$, with non-degenerate caustics $\Q_{\alpha_1}$ and $\Q_{\alpha_2}$ if and only if the following conditions are satisfied:
\begin{itemize}
 \item the caustics are $1$-sheeted hyperboloids, i.e.~$\alpha_1,\alpha_2\in(a_1,a_2)$;
 \item $C_3=0$; and
 \item $C_0+C_1\alpha_2+C_2\alpha_2^2=0$,
\end{itemize}
with $C_0$, $C_1$, $C_2$, $C_3$ being the coefficients in the Taylor expansion about $x=0$:
$$
\frac{\sqrt{(a_1-x)(a_2-x)(a_3-x)(\alpha_1-x)(\alpha_2-x)}}{\alpha_1-x}=C_0+C_1x+C_2x^2+C_3x^3+\dots.
$$

Moreover, according to Theorem \ref{th:winding} the winding numbers of such trajectories satisfy $m_0>m_1>m_2$, with $m_0=4$ and $m_2$ being even.
Thus, $(m_0,m_1,m_2)=(4,3,2)$.
\end{example}

It is interesting to consider the case when the two caustics coincide: $\alpha_1=\alpha_2$.
In that case, the segments of the billiard trajectory are placed along generatrices of the hyperboloid $\Q_{\alpha_1}$, see Figure \ref{fig:4periodic}.

\begin{figure}[h]
\centering
\includegraphics[width=7cm,height=5cm]{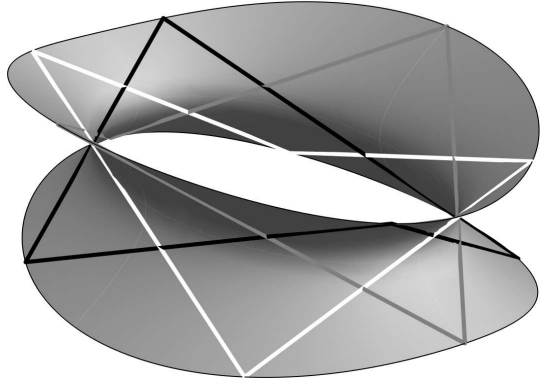}
\caption{Four-periodic trajectories on a hyperboloid.}\label{fig:4periodic}
\end{figure}

\begin{corollary}\label{cor:4hyp}
There exists a $4$-periodic trajectory of the billiard within $\E$, with the segments being parts of generatrices of the confocal $1$-sheeted hyperboloid $\Q_{\alpha_1}$
if and only if
$$
a_1=\frac{a_2a_3}{a_2+a_3}
\quad\text{and}\quad
\alpha_1=a_2+a_3-\sqrt{a_2^2+a_3^2}.
$$
\end{corollary}
\begin{proof}
We will apply Example \ref{ex:n=4} to the case $\alpha_2=\alpha_1$.
Since
$$
C_3=-\frac{ (a_1 a_2-a_1 a_3-a_2 a_3) (a_1 a_2+a_1 a_3-a_2 a_3) (a_1 a_2-a_1 a_3+a_2 a_3)}{16 \left(a_1  a_2 a_3\right)^{5/2}}
$$
and $0<a_1<a_2<a_3$, the coefficient $C_3$ can be equal to zero only if $a_2a_3=a_1a_3+a_1a_2$.
Now, substituting $\alpha_2=\alpha_1$ and $a_1=a_2a_3/(a_2+a_3)$, we get:
$$
C_0+C_1\alpha_2+C_2\alpha_2^2=\frac{\alpha_1 ^2-2(a_2+a_3)\alpha_1+2 a_2 a_3}{2\sqrt{a_2+a_3}}=0.
$$
This is a quadratic equation in $\alpha_1$.
One of its solutions, $a_2+a_3+\sqrt{a_2^2+a_3^2}$, is bigger than $a_3$ so it cannot correspond to a caustic, while the other one,
$a_2+a_3-\sqrt{a_2^2+a_3^2}$, is between $a_1$ and $a_2$.
\end{proof}

\begin{proposition}
Each confocal family of quadrics contains a unique pair of ellipsoid and $1$-sheeted hyperboloid such that there is a $4$-periodic billiard trajectory within the ellipsoid with the segments placed on the hyperboloid.
\end{proposition}
\begin{proof}
We search for $\lambda<a_1$ such that the ellipsoid $\Q_{\lambda}$ from the confocal family satisfies the first condition of Corollary \ref{cor:4hyp}:
$$
a_1-\lambda=\frac{(a_2-\lambda)(a_3-\lambda)}{a_2-\lambda+a_3-\lambda},
$$
which has a unique solution in $(-\infty,a_1)$: $\lambda=a_1-\sqrt{(a_3-a_1)(a_2-a_1)}$.
The corresponding hyperboloid is then uniquely defined from Corollary \ref{cor:4hyp}.
\end{proof}

\subsection{5-periodic trajectories in dimension three}\label{sec:5periodic}

\begin{example}
According to Theorem \ref{th:uslov},
there is a $5$-periodic trajectory of the billiard within ellipsoid $\E$, with non-degenerate caustics $\Q_{\alpha_1}$ and $\Q_{\alpha_2}$ if and only if the following conditions are satisfied:
\begin{itemize}
\item since the period is odd, one of the caustics, say $\Q_{\alpha_1}$, is an ellipsoid, i.e.~$\alpha_1\in(0,c)$; and
 \item $C_3=C_4=0$,
\end{itemize}
with $C_3$, $C_4$ being the coefficients in the Taylor expansion about $x=0$:
$$
\frac{\sqrt{(a_1-x)(a_2-x)(a_3-x)(\alpha_1-x)(\alpha_2-x)}}{\alpha_1-x}=C_0+C_1x+C_2x^2+C_3x^3+\dots.
$$

According to Theorem \ref{th:winding}, the winding numbers $(m_0,m_1,m_2)$ satisfy $m_0=5$, $m_0>m_1>m_2$. Since $m_1$, $m_2$ are even, $(m_0,m_1,m_2)=(5,4,2)$.
The graph of the corresponding polynomial $\hat{p}_5(s)$ is shown in Figure \ref{fig:p542}.
\begin{figure}[h]
\centering
\includegraphics[width=11cm,height=4.15cm]{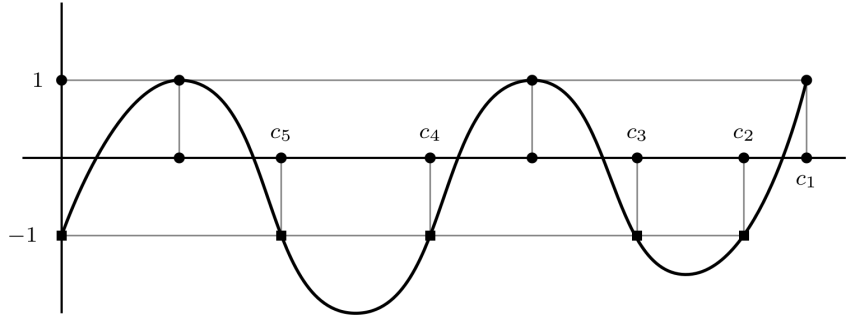}
 \caption{The graph of $\hat{p}_5(s)$. The parameters are $c_1=1/\alpha_1$, $c_2=1/a_3$, $\{c_3,c_4\}=\{1/a_2,1/\alpha_2\}$, $c_{5}=1/a_1$.
 The winding numbers of the trajectory are $(5,4,2)$ and the signature $(0,1,1)$.}\label{fig:p542}
\end{figure}
\end{example}

\subsection{6-periodic trajectories in dimension three}

We saw in Sections \ref{sec:4periodic} and \ref{sec:5periodic} that $4$-periodic and $5$-periodic trajectories the three-dimensional case have uniquely determined winding numbers.
This is not the case with the trajectories of period $6$, which can have winding numbers
$$
(m_0,m_1,m_2)\in\{(6,4,2),(6,5,4),(6,5,2),(6,3,2)\}.
$$

\begin{example}[Winding numbers $(6,4,2)$.]\label{ex:642}
There is a periodic trajectory with winding numbers $(6,4,2)$ of the billiard within ellipsoid $\E$, with non-degenerate caustics $\Q_{\alpha_1}$ and $\Q_{\alpha_2}$ if and only if $6P_0\sim6P_{\infty}$, that is $C_4=C_5=0$, with
$$
\sqrt{(a_1-x)(a_2-x)(a_3-x)(\alpha_1-x)(\alpha_2-x)}=C_0+C_1x+C_2x^2+C_3x^3+\dots.
$$
Moreover, since all winding numbers are even, such trajectories are elliptic $3$-periodic, see Example \ref{ex:d-elliptic}.

The graph of the corresponding polynomial $\hat{p}_6(s)$ is shown in Figure \ref{fig:p642}.
\begin{figure}[h]
\centering
\includegraphics[width=12cm,height=3.96cm]{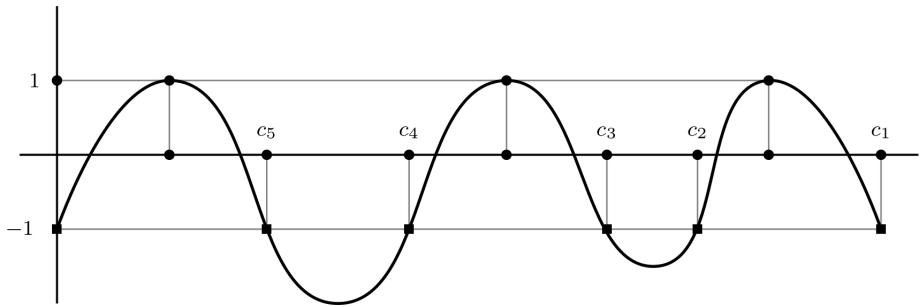}
\caption{The graph of $\hat{p}_6(s)$ corresponding to the winding numbers $(6,4,2)$ and signature $(1,1,1)$. The parameters are $\{c_1,c_2\}=\{1/\alpha_1,1/a_1\}$, $\{c_3,c_4\}=\{1/\alpha_2,1/a_2\}$, $c_5=1/a_3$.
 }\label{fig:p642}
\end{figure}
\end{example}

For trajectories with winding numbers $(6,4,2)$, there are no constraints for types of the caustics.
It is interesting to consider separately the case when the caustics coincide with each other, that is when the trajectories lie on an $1$-sheeted hyperboloid.
An example of such trajectories is shown in Figure \ref{fig:traj642}, with the parameters obtained by application of Example \ref{ex:642}.
Notice that the trajectories are symmetric with respect to the origin.

\begin{figure}[h]
\centering
\includegraphics[width=8cm,height=4.37cm]{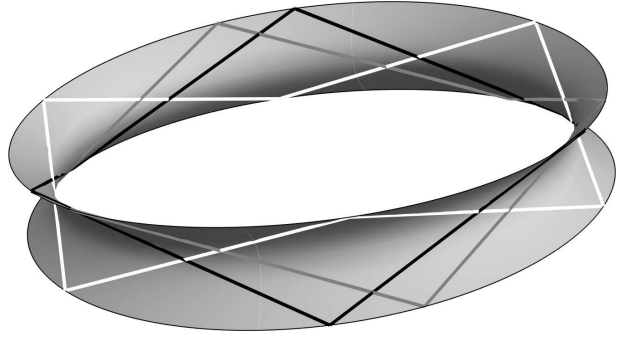}
 \caption{Six-periodic trajectories on a hyperboloid. $m_2=2$ since the trajectories close after winding once around the vertical coordinate axis.
 The parameters are
 $a_1=180-80\sqrt{5}\approx1.11$, $a_2=4$, $a_3=5$, $\alpha_1=\alpha_2=\frac{20}{61}(9-2\sqrt{5})\approx1.48$.}\label{fig:traj642}
\end{figure}

\begin{proposition}\label{prop:cayley642}
There is a periodic trajectory with elliptic period $3$, with the segments being parts of generatrices of a confocal $1$-sheeted hyperboloid
if and only if
$$
\left|\begin{matrix} B_3 & B_4 \\ B_4 & B_5 \end{matrix}\right|=0,
$$
with
$$
\sqrt{(a_1-x)(a_2-x)(a_3-x)}=B_0+B_1x+B_2x^2+B_3x^3+\dots.
$$
\end{proposition}
\begin{proof}
Follows from the condition stated in Example \ref{ex:642}, with $\alpha_1=\alpha_2$, and noticing that
$C_4=\alpha_1B_4-B_3$ and $C_5=\alpha_1B_5-B_4$.
\end{proof}

As a consequence of Proposition \ref{prop:cayley642} and the Cayley's condition, we get

\begin{corollary}
The following three statements are equivalent:
\begin{itemize}
 \item There is a periodic trajectory with elliptic period $3$ of the billiard within ellipsoid $\frac{x_1^2}{a_1}+\frac{x_2^2}{a_2}+\frac{x_3^2}{a_3}=1$, with the segments being parts of generatrices of a confocal $1$-sheeted hyperboloid.
 \item There is a $6$-periodic trajectory within ellipse:
 $\frac{x_2^2}{a_2}+\frac{x_3^2}{a_3}=1$,
 with ellipse $\frac{x_2^2}{a_2-a_1}+\frac{x_3^2}{a_3-a_1}=1$ as caustic.
 \item
 There is a $6$-periodic trajectory within ellipse:
 $\frac{x_1^2}{a_1}+\frac{x_3^2}{a_3}=1$,
 with hyperbola $\frac{x_1^2}{a_1-a_2}+\frac{x_3^2}{a_3-a_2}=1$ as caustic.
\end{itemize}
\end{corollary}

In the remaining three triplets of the winding numbers corresponding to $6$-periodic trajectories, the middle one, $m_1$, is odd.
That means that both caustics in these cases need to be $1$-sheeted hyperboloids.
From Theorem \ref{th:uslov} we get the analytic condition for such trajectories.

\begin{proposition}\label{prop:6odd}
There is a $6$-periodic trajectory with odd frequency number $m_1$ of the billiard within ellipsoid $\E$, with non-degenerate caustics $\Q_{\alpha_1}$ and $\Q_{\alpha_2}$ if and only if both $\Q_{\alpha_1}$ and $\Q_{\alpha_2}$ are $1$-sheeted hyperboloids, that is $\alpha_1,\alpha_2\in(a_1,a_2)$, and any of the following two equivalent conditions satisfied:
\begin{itemize}
 \item[(i)]
$$
\rank\left(
\begin{matrix}
\alpha_1\alpha_2 & 0 & 0 & C_0 & 0 \\
-(\alpha_1+\alpha_2) & \alpha_1\alpha_2 & 0 & C_1 & C_0\\
1 & -(\alpha_1+\alpha_2) & \alpha_1\alpha_2 & C_2 & C_1\\
0 & 1 & -(\alpha_1+\alpha_2)& C_3 & C_2 \\
0 & 0 & 1 & C_4 & C_3\\
0 & 0 & 0 & C_5 & C_4
\end{matrix}
\right)<5,
$$
with
$$
\sqrt{(a_1-x)(a_2-x)(a_3-x)(\alpha_1-x)(\alpha_2-x)}=C_0+C_1x+C_2x^2+C_3x^3+\dots;
$$
 \item[(ii)]
 there are polynomials $p_2(x)$ and $p_1(x)$ of degrees $2$ and $1$ respectively such that:
 \begin{equation}\label{eq:6odd}
 (\alpha_1-x)(\alpha_2-x)p_2^2(x)-(a_1-x)(a_2-x)(a_3-x)p_1^2(x)=x^6.
 \end{equation}
\end{itemize}
\end{proposition}

\begin{proof}
 The algebro-geometric condition for such trajectories is $6P_0+P_{\alpha_1}\sim6P_{\infty}+P_{\alpha_2}$.
The basis of $\LL(6P_{\infty}+P_{\alpha_2})$ is:
$$
1,\ x,\ x^2,\ x^3,\ y,\ \frac{y}{x-\alpha_2}.
$$
We are searching for a non-trivial linear combination $\varphi$ that has a zero of order $6$ at $x=0$ and a simple zero $x=\alpha_1$:
$$
\varphi=A_0+A_1x+A_2x^2+A_3x^3+A_4y+A_5\frac{y}{x-\alpha_2}.
$$
Since $y$ has a zero at $x=\alpha_1$, one root of $A_0+A_1x+A_2x^2+A_3x^3$ is $\alpha_1$, so the condition is equivalent to
\begin{equation}\label{eq:2phi}
(x-\alpha_2)\varphi=(\alpha_1-x)(\alpha_2-x)(A_0'+A_1'x+A_2'x^2)+A_3'y+A_4'xy,
\end{equation}
having a zero of order $6$ at $x=0$.
Part (i) then follows from the Taylor expansion of $(x-\alpha_2)\varphi$ around $x=0$:
\begin{gather*}
(\alpha_1\alpha_2 A_0'+C_0 A_3')
+
\left(-(\alpha_1+\alpha_2)A_0'+\alpha_1\alpha_2A_1'+C_1A_3'+C_0A_4'\right)x
\\
+
(A_0'-(\alpha_1+\alpha_2)A_1'+\alpha_1\alpha_2A_1+C_2A_3'+C_1A_4')x^2
\\
+
(A_1'-(\alpha_1+\alpha_2)A_2'+C_3A_3'+C_2A_4')x^3
\\
+
(A_2'+C_4A_3'+C_3A_4')x^4
+
(C_5A_3'+C_4A_3')x^5+\dots
\end{gather*}

To obtain part (ii), denote $p_2(x)=A_0'+A_1'x+A_2'x^2$, $p_1(x)=A_3'+A_4'x$ and multiply the righthandside of \eqref{eq:2phi} by
$p_2(x)-yp_1(x)$, and divide it by $(\alpha_1-x)(\alpha_2-x)$.
We get that the function:
$$
(\alpha_1-x)(\alpha_2-x)p_2^2(x)-(a_1-x)(a_2-x)(a_3-x)p_1^2(x)
$$
has a zero of order $6$ at $x=0$.
Since that function is a polynomial of degree $6$, we get the requested condition.
\end{proof}

\begin{remark}
Equation \eqref{eq:6odd} from Proposition \ref{prop:6odd} is equivalent to:
$$
r_2^2(s)\left(s-\frac1{\alpha_1}\right)\left(s-\frac1{\alpha_2}\right)
-
r_1^2(s)\cdot s\left(s-\frac1{a_1}\right)\left(s-\frac1{a_2}\right)\left(s-\frac1{a_3}\right)=1.
$$
Now we proceed similarly as in Lemma \ref{lemma:polinomi}.
Denote:
$$
\hat{p}_6(s):=2r_1^2(s)\cdot s\left(s-\frac1{a_1}\right)\left(s-\frac1{a_2}\right)\left(s-\frac1{a_3}\right)+1
=
2r_1^2(s)\cdot \rho_4(s)+1
,
$$
and we get:
$$
\hat{p}_6^2(s):=4r_1^2\rho_4(r_1^2\rho_4+1)+1=4r_1^2 r_2^2 \hat{\Pol}_{6}+1=\hat{q}_3^2\hat{\Pol}_{6}+1.
$$
Finally, we get Pell's equation:
$$
\hat{p}_6^2(s)-\hat{q}_3^2(s)\hat{\Pol}_{6}(s)=1.
$$
\end{remark}

\begin{example}[$6$-periodic trajectories with winding numbers $(6,5,4)$.]
The graph of the corresponding polynomial $\hat{p}_6(s)$ is shown in Figure \ref{fig:p654}.
\begin{figure}[h]
\centering
\includegraphics[width=12cm,height=3.96cm]{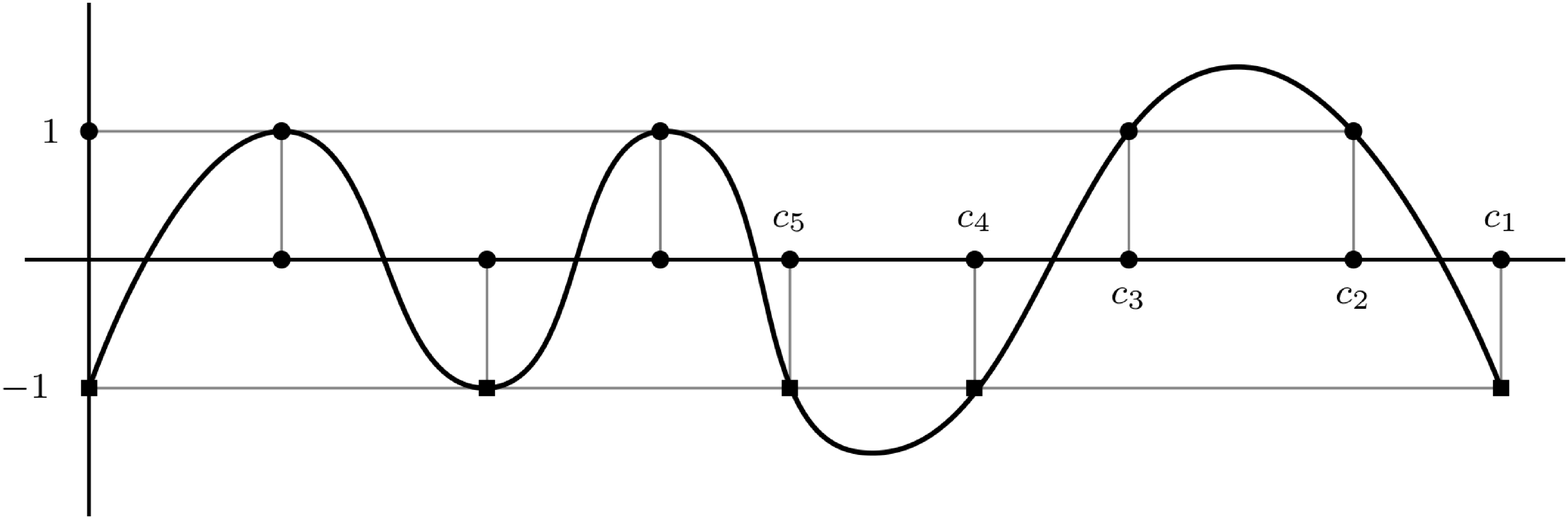}
 \caption{The graph of $\hat{p}_6(s)$ corresponding to the signature $(0,0,3)$ and winding numbers $(6,5,4)$.
 The parameters are $c_1=1/a_1$, $c_2=1/\alpha_1$, $c_3=1/\alpha_2$, $c_4=1/a_2$, $c_{5}=1/a_1$.}\label{fig:p654}
\end{figure}
Several such trajectories, for $\alpha_1=\alpha_2$, are shown in Figure \ref{fig:traj6odd}.
\begin{figure}[h]
\centering
\includegraphics[width=8cm,height=6.84cm]{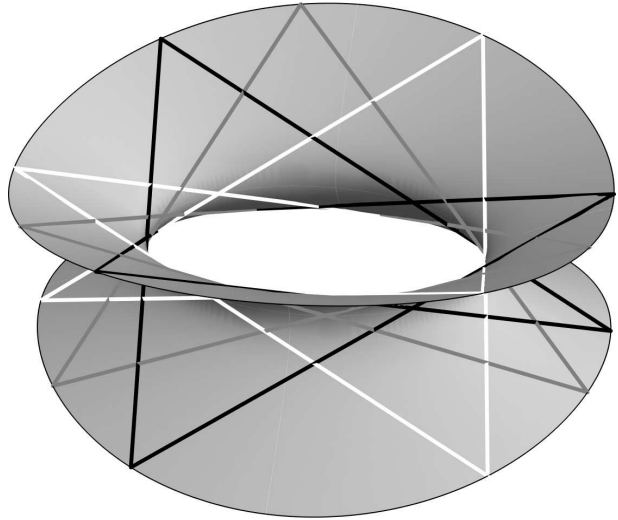}
 \caption{Six-periodic trajectories on a hyperboloid. $m_2=4$ since the trajectories close after winding twice around the vertical coordinate axis.
 The parameters are $a_1\approx3.303$, $a_2=4$, $a_3=5$, $\alpha_1=\alpha_2\approx3.5$.}\label{fig:traj6odd}
\end{figure}

\end{example}

\begin{example}[$6$-periodic trajectories with winding numbers $(6,5,2)$.]

The graph of the corresponding polynomial $\hat{p}_6(s)$ is shown in Figure \ref{fig:p652}.
\begin{figure}[h]
\centering
\includegraphics[width=12cm,height=3.96cm]{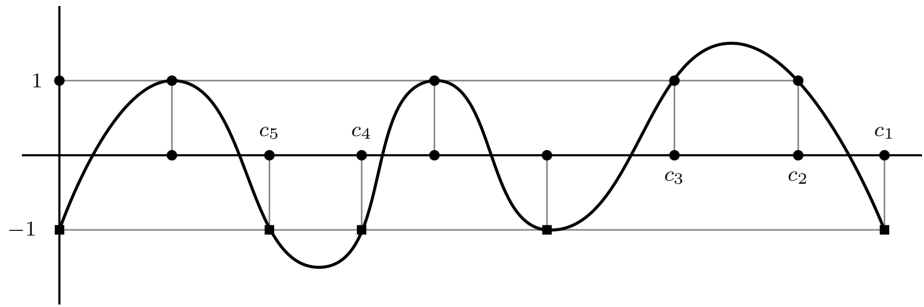}
 \caption{The graph of $\hat{p}_6(s)$ corresponding to the winding numbers $(6,5,2)$ and signature $(0,2,1)$. The parameters are $c_1=1/a_3$, $c_2=1/\alpha_1$, $c_3=1/\alpha_2$, $c_4=1/a_2$, $c_{5}=1/a_1$.}\label{fig:p652}
\end{figure}
\end{example}

\begin{example}[$6$-periodic trajectories with winding numbers $(6,3,2)$.]
The graph of the corresponding polynomial $\hat{p}_6(s)$ is shown in Figure \ref{fig:p632}.
\begin{figure}[h]
\centering
\includegraphics[width=12cm,height=3.96cm]{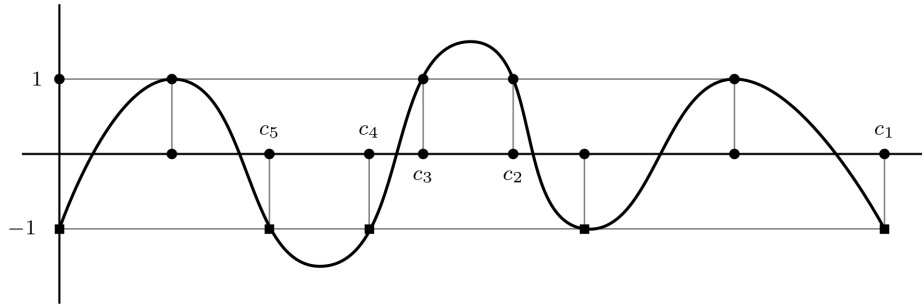}
 \caption{The graph of $\hat{p}_6(s)$ corresponding to winding numbers $(6,3,2)$ and signature is $(2,0,1)$. The parameters are  $c_1=1/a_3$, $c_2=1/\alpha_1$, $c_3=1/\alpha_2$, $c_4=1/a_2$, $c_{5}=1/a_1$.}\label{fig:p632}
\end{figure}
\end{example}

\subsection*{Scknowledgements}
The research  was supported
by the Australian Research Council, Discovery Project 190101838 \emph{Billiards within quadrics and beyond} and by the
 Serbian Ministry of Education, Science, and Technological Development,
Project 174020 \emph{Geometry and Topology of Manifolds, Classical
Mechanics, and Integrable Dynamical Systems}.
M.~R.~is grateful to Holger Dullin for discussions.
Both autors thank the referee for careful reading and useful suggestions.

\end{document}